\documentclass[12pt]{amsart}
\usepackage{bbm,amssymb,amsmath, accents, rlepsf, stmaryrd}

  \usepackage{hyperref}

\input xy
\xyoption{all}

\usepackage{color}

\newcommand{\ds}{\displaystyle}

\newtheorem{theorem}{Theorem}[section]
\newtheorem{lemma}[theorem]{Lemma}
\newtheorem{proposition}[theorem]{Proposition}

\theoremstyle{definition}
\newtheorem{definition}[theorem]{Definition}
\newtheorem{remark}[theorem]{Remark}
\newtheorem{example}[theorem]{Example}

\numberwithin{equation}{section}

\newcommand{\gras}[1]{{\mathbb #1}}

\newcommand{\R}{\gras{R}}
\newcommand{\C}{\gras{C}}
\newcommand{\bS}{\gras{S}}
\newcommand{\bD}{\gras{D}}

\newcommand{\cU}{\mathcal{U}}
\newcommand{\cN}{\mathcal{N}}

\newcommand{\cM}{\mathcal{M}}

\newcommand{\codim}{\mbox{codim}}

\def\elem(#1,#2){  \{ \frac{#1}  {\overline {\ #2\ }} \} }

\title[Generalized plumbings and Murasugi sums]
    {Generalized plumbings and Murasugi sums}
\author{Burak Ozbagci}
    \address{Department of Mathematics, Ko\c{c} University, Rumelifeneri Yolu,
         34450, Sariyer, Istanbul, Turkey}
     \email{bozbagci@ku.edu.tr}
\author{Patrick Popescu-Pampu}
   \address{Universit{\'e} Lille 1, UFR de Maths., B\^atiment M2\\
     Cit\'e Scientifique, 59655, Villeneuve d'Ascq Cedex, France.}
   \email{patrick.popescu@math.univ-lille1.fr}
\date{}

\begin{document}

\begin{abstract}
   We propose a generalization of the classical notions of plumbing and Murasugi
   summing operations to smooth manifolds of arbitrary dimensions, so that
   in this general context Gabai's credo
   ``the Murasugi sum is a natural geometric operation'' holds. In particular, we
   prove that the sum of the pages of two open books is again a page
   of an open book and that there is an associated summing
   operation of Morse maps.  We conclude with several open questions
   relating this work with singularity theory and contact topology.
\end{abstract}

\maketitle

\begin{center}
    {\bf This paper appeared in Arnold Math. Journal (2016) 2:69-119. 
         DOI 10.1007/s40598-015-0033-3}
\end{center}

\tableofcontents

\section{Introduction}
\label{sec:Intro}

Around 1960, Milnor and Mumford introduced independently particular
cases of an operation which builds new manifolds with boundary from
given ones: ``\emph{plumbing}''. Milnor used this operation to construct
exotic spheres in higher dimensions  and Mumford in order to
describe the boundaries of nice neighborhoods of isolated singular
points on complex surfaces.

Around the same time, Murasugi defined an analogous operation on
Seifert surfaces of links in the $3$-sphere. This operation was done
on embedded objects rather than abstract ones. Nevertheless, this
operation agrees with (a slight generalization of) the plumbing
operation on the embedded surfaces.

In the mid-seventies, Stallings introduced the name of ``\emph{Murasugi sum}''
for the operation above, and he showed that \emph{the Murasugi
sum of two pages of open books is again the page of an open book}.
Several years later, Gabai proved that Murasugi sum preserves other
properties of surfaces embedded in $3$-manifolds, and summarized the
general philosophy behind such results by the credo ``\emph{Murasugi
sum is a natural geometric operation}'' (see Gabai \cite{G 83}, \cite{G 83bis}, 
\cite{G 85}, \cite{G 86}).

In the mid-eighties, Lines proved an analog of Stallings' theorem
for special types of open books in higher dimensional spheres, after
having extended to that context the operation of Murasugi sum.

Details about the previous historical facts may be found in Sections
\ref{sec:classop} and \ref{sec:geomproof} of our paper.
\medskip

The effect of the Murasugi sum on the hypersurfaces under scrutiny
is to \emph{plumb} them, that is, roughly speaking, to identify by a special diffeomorphism two
balls embedded in them, in such a way that the result is again a
manifold with boundary.

The aim of this paper is \emph{to identify the most general
operation of plumbing in arbitrary dimensions, which allows
one to extend the classical operation of Murasugi sum, such
that Gabai's credo still holds}.

Our main result (see Theorem \ref{thm:genstal}) is that \emph{an
analog of Stallings' theorem holds if the plumbing operation is
generalized by allowing the gluing of two manifolds with boundary
through any diffeomorphism of compact full-dimensional submanifolds,
provided that the result is again a manifold with boundary}.

In particular, we never impose orientability hypotheses. Instead,
throughout  the paper the crucial assumptions are about
\emph{coorientability} of hypersurfaces. Moreover, we work with
fixed coorientations. As those coorientations are present in the
absence of any orientations on the ambient manifold or on the
hypersurface, we work in a slightly non-standard context. This
obliges us to give careful definitions of all the objects we
manipulate, by lack of a convenient source in the literature.

An important message of our work is that it is much easier to prove
that generalized Murasugi sums conserve geometric properties
(illustrating Gabai's credo) if the fundamental notion of \emph{sum}
is defined on special types of \emph{cobordisms}. In fact, the most
difficult result of our work from the technical viewpoint
(Proposition \ref{prop:samedef}) states that our generalization of
the Murasugi sum to arbitrary dimensions coincides with another
definition given in terms of cobordisms.

We believe that, combining our new operations with those
explored in \cite{KN 77} and \cite{NR 87}, one will get a better
understanding of the differential topology of singularities.

\medskip

Let us describe the structure of the paper.

In Section \ref{sec:classop} we sketch the historical evolution of
the notions of plumbing and Murasugi sum, through the works of
Milnor, Mumford, Murasugi, Stallings, Gabai and Lines. We
quote from the original sources, in order to allow the reader to
compare easily those classical constructions to ours.

In Section \ref{sec:geomproof} we explain Gabai's geometric proof of
Stallings' theorem. We describe a variant of his proof given by
Giroux and Goodman and  give a second interpretation of it as explained by
Etnyre.

In Section \ref{sec:notconv} we explain our basic conventions about \emph{coorientations}
of hypersurfaces in manifolds with boundary (see Definition \ref{def:coorient}),
their \emph{sides} and \emph{collar neighborhoods}
(see Definition \ref{def:sides}).

In Section \ref{sec:cobcorn} we set up our notation for
\emph{cobordism of manifolds with boundary} (see Definition
\ref{def:cobcorners}), which is essential for our approach, mainly
through its special case of \emph{cylindrical cobordisms} (see
Definition \ref{def:cylcob}). Cobordisms of manifolds with
boundary may also be \emph{composed}, just like usual cobordisms. In
the following sections, for concision,  we simply speak about
\emph{cobordisms} instead of \emph{cobordism of manifolds with
boundary}.

In Section \ref{sect:Seifsect} we describe the notions of
\emph{Seifert hypersurfaces} (see Definition \ref{def:seifhyp}) and
\emph{open books} (see Definition \ref{def:openbook}) and establish
the equivalence of these notions with some special types of
cobordisms.

In Section \ref{sec:absum} we introduce our generalizations of the
classical notions of \emph{plumbing} and \emph{Murasugi sum}. We
call them \emph{abstract} and \emph{embedded summing} respectively
(see Definitions \ref{def:absum} and \ref{def:embsumpatch}). For the
latter, the hypersurfaces to be summed are not assumed to be
coorientable, but only the identified \emph{patches} (see Definition
\ref{def:patch}) are assumed to be \emph{cooriented}. We show that
embedded summing is an associative but in general non-commutative
operation (see Proposition \ref{prop:propembsum}).

In Section \ref{sec:embsum} we introduce a supplementary structure
on cylindrical cobordisms, \emph{stiffenings}, which exist and are
unique up to isotopy, but which are not canonical. Such structures
are essential for the proofs presented in
Section~\ref{sect:natgeomop}. We also define a summing
operation on stiffened cylindrical cobordisms (see Definition
\ref{def:embsum}).

In Section \ref{sect:natgeomop} we show that, under the assumption
that the hypersurfaces which are to be summed in an embedded way are
\emph{globally cooriented}, the operation of embedded  summing may
be reinterpreted as a summing operation on cylindrical cobordisms
(see Proposition \ref{prop:samedef}). Our generalization of
Stallings' theorem (see Theorem \ref{thm:genstal}) is obtained then
easily by working with a stiffening adapted to the open books under
scrutiny. We also formulate an extension of this theorem to what we
call \emph{Morse open books} (see Definition
\ref{def:morseopenbook}).

Finally, in Section \ref{sec:ques} we list several open questions.
Some of them concern the relations of open books with singularity
theory and contact topology. For this reason, we begin that section by
recalling briefly the basics of those relations. We hope that
this work will be useful in particular to the researchers interested in the
topology of singular spaces and to those interested in the topology
of contact manifolds.

\medskip
{\bf Acknowledgments.} We thank the referee for his/her careful reading
   of the submitted version of this paper. The second author was supported
   by Labex CEMPI (ANR-11-LABX-0007-01) and by the grant
   ANR-12-JS01-0002-01 SUSI. Both authors were supported
   by the joint CNRS-Tubitak project 113F007, titled {\em Topology of surface
   singularities} (2013-14).

\bigskip
\section{Plumbing and Murasugi sums in the literature}
\label{sec:classop}

In this section we recall the classical notions of plumbing, as
defined by Milnor and by Mumford, as well as Murasugi's original
construction, its extensions by Stallings and Gabai to more general
$3$-dimensional operations and by Lines to arbitrary dimension.
\medskip

In \cite[p.71]{M 59}, Milnor constructed for any $k \geq 1$ a
$(2k-1)$-connected manifold-with-boundary $M_k$ of dimension $4k$
whose intersection form in dimension $2k$ has the following matrix:
   \[   \left(  \begin{array}{cccccccc}
              2 & 1 &  &  &  &   &  &  \\
              1 & 2 & 1  &  &  -1 &   &  &  \\
               & 1 & 2  & 1  &  &   &  &  \\
               &  & 1 & 2 & 1 &   &  &  \\
               & -1 &  & 1 & 2 & 1  &  &  \\
               &  &  &  & 1 &  2 & 1 &  \\
               &  &  &  &  & 1  & 2 & 1 \\
               &  &  &  &  &   & 1 & 2
        \end{array}   \right)  \]
in an appropriate basis, where the missing entries are $0$.
The determinant of this matrix is $1$, which ensures that
the boundary of the constructed manifold is homeomorphic to a
sphere. Milnor showed that this boundary  generated the cyclic group
of $7$-dimensional homotopy spheres which bound parallelizable manifolds.

In order to construct $M_k$, Milnor started from two transversal
copies of the sphere $\bS^{2k}$ inside $\bS^{2k} \times \bS^{2k}$,
intersecting in exactly two points, and having self-intersections
$+2$: the diagonal and its image by the map $1 \times \alpha$, where
$\alpha: \bS^{2k} \to \bS^{2k}$ denotes in his words the ``twelve
hour rotation which leaves the north pole fixed, and satisfies
$\alpha(x) = -x$ for $x$ on the equator''.

He took the universal cover $\tilde{U}$ of a tubular neighborhood
$U$ of the union $X$ of the two spheres, and looked at the total preimage
$\tilde{X}$ of $X$ inside $\tilde{U}$. He could easily find
 in $\tilde{U}$  a sequence:
  \[ T_1 \cup T_1' \cup   T_2 \cup T_2' \cup T_3 \cup T_3' \cup T_4 \cup T_4'   \]
of tubular neighborhoods of eight $(2k)$-dimensional spheres of
$\tilde{X}$ intersecting in a chain, whose intersection matrix is
isomorphic to the one given above, except that the two $-1$'s are
replaced by $0$-s. Milnor explains at this point:

\begin{quote}
    ``To correct this intersection matrix it is necessary to introduce an intersection
    between $T_1'$ and $T_3$, so as to obtain an intersection number $-1$.
    Choose a rotation of $\bS^{2k} \times \bS^{2k}$ which carries a region of
    $T'$ near the ``equator'' onto a region of $T$ near the ``equator'', so as to
    obtain an intersection number of $-1$. Matching the corresponding regions
    of $T_1'$ and $T_3$, we obtain a topological manifold $W_2$, with the required
    intersection matrix.''
\end{quote}

We note that $W_2$  is not the final manifold in Milnor's
construction, but this is not so important for our purposes. It is
this ``matching'' of regions which was later named
``\emph{plumbing}'', following a denomination introduced for a
related object by Mumford \cite{M 61}.

Mumford's problem in \cite{M 61} was to study the topology of the
boundary of a ``tubular neighborhood'' of a reducible compact
complex curve in a smooth complex surface. He assumed that the
irreducible components $E_i$ of the curve are \emph{smooth} and he
described the boundary $M$ of their union as the result of a
cut-and-paste operation done on the boundaries $M_i$ of tubular
neighborhoods of the individual $E_i$'s. One  first has to
cut some solid tori from the $M_i$'s  and then glue pairwise
collar neighborhoods of the boundary components created in this way.
He described those collar neighborhoods as ``\emph{standard plumbing
fixtures}'' (see \cite[Page 8]{M 61}). The term ``\emph{plumbing}''
was brought to this context! Later, it was used as a name for two
different but related operations:

 \begin{itemize}
      \item[$\bullet$] following Mumford, a cut-and-paste operation used to describe the
         \emph{boundary} of a tubular neighborhood of a union of submanifolds
         of a smooth manifold, intersecting generically (see \cite{N 81} and \cite{P 07});
      \item[$\bullet$] following Milnor, a purely pasting operation used to describe
      the tubular neighborhoods themselves.
 \end{itemize}

 One of the first definitions of this operation in a textbook is to be found
 in \cite[Chapter 8]{HNK 71}. Let us quote from it the definition of the plumbing
 of two $n$-disc bundles (see Figure \ref{S}, reproduced from the same book):

\begin{quote}
 \begin{definition}  \label{def:plumb}
      ``Let $\xi_1 =(E_1, p_1, \bS_1^n)$ and $\xi_2 =(E_2, p_2, \bS_2^n)$ be two oriented
      $n$-disc bundles over $\bS^n$. Let $D_i^n \subset \bS_i^n$ be embedded $n$-discs
      in the base spaces and let:
          \[ f_i : D_i^n \times D^n \to E_i | D_i^n \]
      be trivializations of the restricted bundles $E_i | D_i^n$ for $i = 1,2$.
      To {\bf plumb} $\xi_1$ and $\xi_2$ we take the disjoint union
      of $E_1$ and $E_2$ and identify the points $f_1(x,y)$ and $f_2(y,x)$
      for each $(x,y) \in D^n \times D^n$.''
 \end{definition}
 \end{quote}

 \begin{figure}[ht]
 \relabelbox \small {
 \centerline{\epsfbox{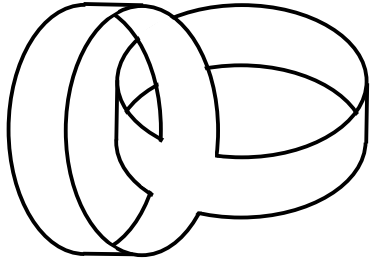}}}
\endrelabelbox
       \caption{Plumbing of two $n$-disc bundles according to \cite{HNK 71}}
       \label{S}
\end{figure}

 It was Hirzebruch \cite{H 63} who related Milnor's and Mumford's constructions:

 \begin{quote}
    ``$M(E_8)$ was constructed by ``plumbing'' $8$ copies of the circle bundle
    over $\bS^{2k}$ with Euler number $-2$. By replacing this basic constituent by
    the tangent bundle of $\bS^{2k}$ one obtains a manifold $M^{4k -1}(E_8)$
    of dimension $4k-1$. This carries a natural differentiable structure.
    For $k \geq 2$ it is homeomorphic to $\bS^{4k-1}$, but not diffeomorphic
    (Milnor sphere).''
 \end{quote}

  Here Hirzebruch proposed an alternative construction of a generator
  of the group of homotopy spheres of dimension $7$, as the intersection
  matrix of the $E_8$ diagram is simpler than the one considered by Milnor
  in \cite{M 59}. In fact, Milnor presented later in \cite{M 64} Hirzebruch's
  ``plumbing'' construction according to the  $E_8$ diagram
  rather than his initial construction.

The operation of ``plumbing'' was generalized from $n$-disc bundles over
$n$-dimensional spheres to arbitrary $n$-dimensional manifolds as
base spaces, the identifications of $f_1(x,y)$ and $f_2(-y,x)$ being also
allowed (see, for example, Browder's book \cite[Section V.2]{B 72}).
 Nevertheless, what remained unchanged was the structure
of the subbundles to be patched: products $D^n \times D^n$ of
$n$-dimensional discs.

\medskip
Now let us turn our attention to the related notion of
\emph{Murasugi sum}. We quote below the original construction by
Murasugi \cite[p.545]{M 63}, illustrating it in Figures \ref{Z} and
\ref{Y} by drawings which are similar to Murasugi's original ones:

\begin{quote}
   ``Let us consider an orientable surface $F$ in $\bS^3$ [...] consisting of two
   disks $D_1, D_2$ to which $n$ bands $B_1, B_2, ..., B_n$ are attached.
   All $B_i$ are twisted once in the same direction, and the bands are
   pairwise disjoint and do not link one another. Let us call $F$ a
   \emph{primitive s-surface} of type $(n, \epsilon)$, where
   $\epsilon = \pm 1$ according as the twisting is right-handed or left-handed.
   [...]

   Consider two primitive $s$-surfaces $F$ and $F'$ in $\bS^3$ of type
   $(n, \epsilon)$ and $(m, \eta)$. Take two disks, $D_1$ and $D_1'$ say,
   from each $F$ and $F'$ and identify them so that the resulting orientable surface
   $\tilde{F} = F \cup F'$ spans a link, and that $\tilde{F} - F$ and
   $\tilde{F} - F'$ are \emph{separated}, i.e. there exists a $2$-sphere
   $S$ in $\bS^3$ such that $S \cap \tilde{F} = D_1 (= D_1')$ and each component of
   $\bS^3 - S$ contains points of $\tilde{F} -  D_1$. [...] $\tilde{F}$ will be called
   an \emph{s-surface}. [...] In general, by an \emph{s-surface} is meant an orientable
   surface obtained from a number of primitive s-surfaces by identifying their disks
   in this manner.''

\begin{figure}[ht]
  \relabelbox \small {
  \centerline{\epsfbox{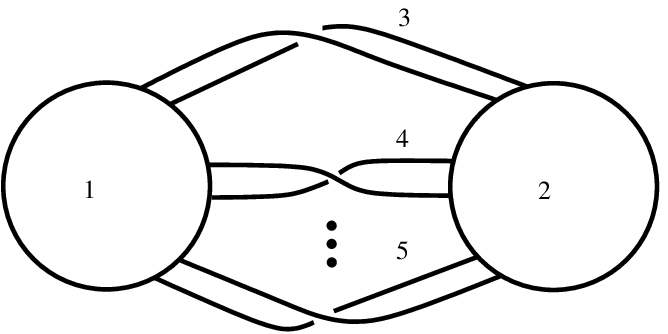}}}

\relabel{1}{$D_1$}

\relabel{2}{$D_2$}

\relabel{3}{$B_1$}

\relabel{4}{$B_2$}

\relabel{5}{$B_n$}

\endrelabelbox
        \caption{Primitive s-surface of type $(n, 1)$, whose boundary is the $(-2,n)$-torus link}
        \label{Z}
\end{figure}

\begin{figure}[ht]
  \relabelbox \small {
  \centerline{\epsfbox{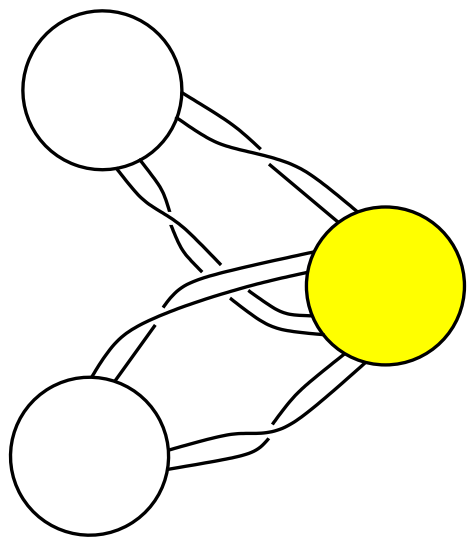}}}

\endrelabelbox
        \caption{Disks in primitive s-surfaces of type $(2, 1)$ and of type
            $(2, -1)$ are identified to give a Seifert surface for a figure-eight knot.}
        \label{Y}
\end{figure}

\end{quote}

The ``\emph{primitive s-surfaces}'' used by Murasugi are
\emph{fiber-surfaces}, that is, they appear as the pages of some
open books in $\bS^3$ (see Definition \ref{def:openbook} below).
In \cite[p.56]{S 78}, Stallings generalized
Murasugi's construction to arbitrary fiber-surfaces as follows:

\begin{quote}
    ``Consider two oriented fibre surfaces $T_1$ and $T_2$. On $T_i$ let $D_i$
    be $2$-cells, and let $h : D_1 \to D_2$ be an orientation-preserving
    homeomorphism such that the union of $T_1$ and $T_2$ identifying $D_1$ with
    $D_2$ by $h$ is a $2$-manifold $T_3$. That is to say:
      \begin{equation}  \label{eq:plumbrel}
            h(D_1 \cap \mbox{Bd } T_1) \cup (D_2 \cap \mbox{Bd } T_2) =  \mbox{Bd } D_2.
       \end{equation}
      [Here $\mbox{Bd } X$ denotes the boundary of $X$].

      We can realize $T_3$ in $\bS^3$ as follows: Thicken $D_1$ on the positive
      side of $T_1$, to get a $3$-cell, whose complementary $3$-cell $E_1$ contains
      $T_1$ with $T_1 \cap \:  \mbox{Bd } E_1 = D_1$ and with negative side of
      $T_1$ contained in the interior of $E_1$. Likewise, there is a $3$-cell
      $E_2$ containing  $T_2$ with $T_1 \cap \:  \mbox{Bd } E_1 = D_1$ and with
      the positive side of  $T_2$ contained in the interior of $E_2$. The homeomorphism
      $h: D_1 \to D_2$ extends to $h : \mbox{Bd } E_1 \to \mbox{Bd } E_2$.
      The union of $E_1$ and $E_2$, identifying their boundaries by $h$
      - this is $\bS^3$ - contains $T_3$ as $T_1 \cup T_2$. We say
      $T_3$ is obtained from $T_1$ and $T_2$ by \emph{plumbing}.''
\end{quote}

The main result of Stallings' paper is:

\begin{theorem} \label{thm:stal}
    If $T_1$ and $T_2$ are fibre surfaces, so is $T_3$.
\end{theorem}

This shows in particular that the s-surfaces of Murasugi are fibre surfaces.
Note that, Stallings' definition of (embedded) ``plumbing'' applies to any
oriented surfaces in $\bS^3$, not only to fibre surfaces.

In \cite[p.132]{G 83}, Gabai coined the name ``Murasugi sum'' for a
slightly restricted operation:

\begin{quote}
   ``The oriented surface $R \subset \bS^3$ is a Murasugi sum of compact
   oriented surfaces $R_1$ and $R_2$ in $\bS^3$ if:
      \begin{enumerate}
          \item $R = R_1 \cup_D R_2$, $D = 2n$ gon
          \item  $R_1 \subset B_1, R_2 \subset B_2$ where
             $B_1 \cap B_2 = S, S$ a $2$-sphere,
                $B_1 \cup B_2 = \bS^3$ and $R_1 \cap S = R_2 \cap S = D$. ''
      \end{enumerate}
\end{quote}

As remarked by Gabai, this definition extends immediately to an
operation on oriented surfaces in arbitrary oriented $3$-manifolds.

Note that in the definition above, the way that $D$ is
embedded in $R_1 \cup_D R_2$ is not explicitly stated, but in
Gabai's drawing \cite[Figure 1]{G 83} the edges on the boundary of
the $2n$-gon $D$ appear as arcs included \emph{alternately} in the
interior of $R_1$ and in the interior of $R_2$. Thus we may deduce
that this slightly restricted operation is what Gabai had in mind
from the fact that $\partial D$ gets a structure of a polygon with
an even number of edges from its embedding in both $R_1$ and $R_2$.

In \cite{R 98} Rudolph called this second, more restrictive interpretation
of the summing operation, ``Murasugi sum'' and reserved the name
``Stallings plumbing'' for Stallings' more general definition. Changing his
notations to those of Stallings' paper, in order to be able to make
reference to the identity (\ref{eq:plumbrel}), let us quote his comparison
of the two definitions:

\begin{quote}
    ``On its face, Stallings plumbing is a strict generalization of Murasugi sum,
    [...] its seemingly special case in which [...] (\ref{eq:plumbrel}) is
    supplemented by:
          \begin{equation}  \label{eq:Murarel}
            h(D_1 \cap \mbox{Bd } T_1) \cap (D_2 \cap \mbox{Bd } T_2) =
                 \partial(D_2 \cap \mbox{Bd } T_2).
       \end{equation}
      In fact, however, it is easy to see that (up to ambient isotopy) every
      Stallings plumbing is a Murasugi sum of the same plumbands. The distinction
      is nonetheless useful and will be maintained here."
\end{quote}

The fact that the more general notion of ``Stallings plumbing'' is
``nonetheless useful'', even if it describes the same objects as the
``Murasugi sum'' may be seen already from the first application of
Theorem \ref{thm:stal} given by Stallings in his paper
(\cite[Theorem 2]{S 78}):

\begin{theorem}  \label{thm:hombraid}
   The oriented link $\hat{\beta}$ obtained by closing a homogeneous braid
   $\beta$ is fibered.
\end{theorem}

A \emph{homogeneous braid} is described by a word in the standard
presentation of the braid groups, such that each generator appears
always with exponents of the same sign. In the special case in which
the generators are always positive, one obtains the usual notion of
\emph{positive braid}. Stallings' proof considers the Seifert
algorithm for constructing a Seifert surface applied to the diagram
of the link $\hat{\beta}$ associated to the given word. The Seifert
surface appears constructed as a finite sequence of disks situated
in parallel planes, successive disks being connected by twisted
bands. The condition of homogeneity says that \emph{all} the bands
between two given successive disks are twisted in the same sense
(see Figure \ref{X}).
One recognizes therefore an s-surface of Murasugi, which is in
general a ``Stallings plumbing'' in Rudolph's sense, but not a
``Murasugi sum'' in Gabai's sense.

\begin{figure}[ht]
  \relabelbox \small {
  \centerline{\epsfbox{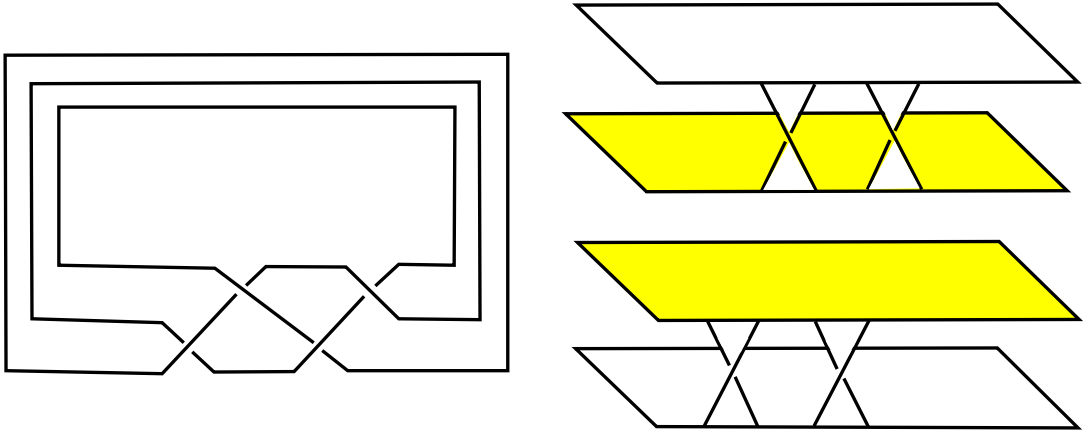}}}

\endrelabelbox
        \caption{On the left: the figure eight knot $\hat{\beta}$ which is the
        closure of the homogeneous braid
        $\beta=\sigma_1^{-1}\sigma_2\sigma_1^{-1}\sigma_2$. On the
        right: the top two disks with twisted bands connecting them
        form a primitive s-surface of type $(2,-1)$, while the lower two disks with twisted
        bands connecting them
        form a primitive s-surface of type $(2,1)$. By gluing these primitive s-surfaces in the
        obvious way, we get a Seifert surface for $\hat{\beta}$. Compare with Figure~\ref{Y}.}
        \label{X}
\end{figure}

\medskip
For a special type of higher dimensional  hypersurfaces in spheres,
a generalization of Murasugi summing was
studied by Lines in a series of papers \cite{L 85, L 86, L 87}. Here are
the definitions he used:

\begin{definition} \label{def:simpleknot}
    A {\bf knot}  $K \subset  \bS^{2k+1}$ is a $(k-2)$-connected oriented
    $(2k-1)$-dimensional submanifold. A {\bf Seifert surface}
    for $K$ is a compact oriented hypersurface of $\bS^{2k+1}$
    with boundary $K$. The knot $K$ is called {\bf simple} if it admits a
    $(k-1)$-connected Seifert surface.
\end{definition}

The following definition appeared in \cite[Section 2]{L 85}:

\begin{definition} \label{def:msumhigh}
Let $K_1$ and $K_2$ be two simple knots in $\bS^{2k+1}$ bounding
$(k-1)$-connected Seifert surfaces $F_1$ and $F_2$ respectively.
Suppose that $\bS^{2k+1}$ is the union of two balls $B_1$ and $B_2$
with a common boundary which is a $(2k)$-sphere $S$.
Let $\psi: \bD^k \times \bD^k \to S$ be an embedding such that:

\begin{enumerate}
   \item  $F_1 \subset B_1, \; $  $ F_2 \subset B_2$;
   \item $F_1 \cap S =F_2 \cap S= F_1 \cap F_2=\psi(\bD^k \times
        \bD^k)$;
   \item $\psi(\partial \bD^k \times \bD^k)= \partial F_1 \cap \psi(\bD^k \times
        \bD^k)$ and  $\psi(\partial \bD^k \times \partial \bD^k)=
        \partial F_2 \cap \psi(\bD^k \times
        \bD^k)$.
 \end{enumerate}
Then the submanifold $F:= F_1 \cup F_2 \subset \bS^{2k+1}$,
after smoothing the corners, is said to be {\bf obtained by plumbing together}
the Seifert surfaces  $F_1$ and $F_2$.
 \end{definition}

In \cite[Proposition 2.1]{L 85}, Lines proved that Theorem
\ref{thm:stal} extends to this context.  His proof is algebraic, not
geometric. In the sequel, we will extend his definition, dropping
any hypothesis on the topology of the pages and of the ambient
manifold
 (see Definition \ref{def:embsumpatch}), and we  will show, through a
geometric proof,  that Theorem \ref{thm:stal} extends also to this
more general context (see Theorem \ref{thm:genstal}).

\bigskip
\section{A geometric proof of Stallings' Theorem}
\label{sec:geomproof}

For the sake of completeness,  we include here a geometric proof of
Theorem~\ref{thm:stal}, for the most frequently used case in the
literature, where the plumbing region is just a rectangle ($n=2$ in
Gabai's Murasugi sum). The principle of the proof below is due to
Gabai \cite[p.139-141]{G 83}, although we will present here another
formulation of his proof which appeared more recently in
\cite[p.101]{GG 06}, using the language of open books (see
Definition~\ref{def:openbook}), rather than fibered surfaces or
foliations.
\medskip

First we prepare a local model of a neighborhood of a properly
embedded arc in the page of  an open book in an arbitrary
$3$-manifold  as follows. Set:
        \[ \tilde{K}=\{(x,y,z) \in \mathbb{R}^3 \;| \; x=\pm{1},  \;\; y=0\}  \]
and let $\tilde{\theta} :  \mathbb{R}^3 \setminus \tilde{K} \to
\mathbb{S}^1$ be the map defined by:
        \[ \tilde{\theta}(x,y,z) = \arg \big(\dfrac{1+x+iy}{1-x-iy}\big)=\arg
                    (1-x^2-y^2+2iy). \]
As $\tilde{\theta}$ does not depend on the $z$-coordinate,
for each $t \in \mathbb{S}^1$, the preimage $\tilde{\theta}^{-1} (t)
$ can be described as the intersection $\tilde{\theta}^{-1} (t) \cap
\{z=0\}$, translated invariantly in the $z$-direction. Therefore, to
visualize  $\tilde{\theta}^{-1} (t) $, it suffices to understand
$(\tilde{\theta}|_{\{z=0\}}) ^{-1} (t)$ which is the preimage of a
ray
 starting from the origin in the complex plane under the homography 
 defined by the equation $w =
\dfrac{1+u}{1-u}$,
 where $u = x + iy$. Since homographies preserve the circles, each such preimage
 is included in some  circle on the $xy$-plane. Using the last equality above,
 it is possible to see that for each $t \neq 0, \pi \in \mathbb{S}^1$, the preimage 
 $(\tilde{\theta}|_{\{z=0\}}) ^{-1} (t)$ is  an
 open arc of a circle passing through
 $(1, 0)$ and $(-1,0)$, as depicted in Figure~\ref{opbook}.  For $t=0$ and $t=\pi$, 
 these preimages are given by the segment $(-1,1)$
 and $\R \setminus [-1, 1]$ on the $x$-axis, respectively.

It follows that, for $t \neq \pi \in \mathbb{S}^1$, the union
$\tilde{\theta}^{-1} (t) \cup \tilde{K}$ is a connected infinite
strip parallel to the $z$-axis, while $\tilde{\theta}^{-1} (\pi)
\cup \tilde{K}$ consists of two connected components. Therefore,
$\tilde{\theta}$ is not a locally trivial fibration over
$\mathbb{S}^1$ (and hence it does not define an open book on
$\mathbb{R}^3$), but nevertheless, $\tilde{\theta}^{-1} (t) \cup
\tilde{K}$ is still called a ``page'' of $\tilde{\theta}$, since it
gives a ``piece" of an open book.

\begin{figure}[ht]
  \relabelbox \small {
  \centerline{\epsfbox{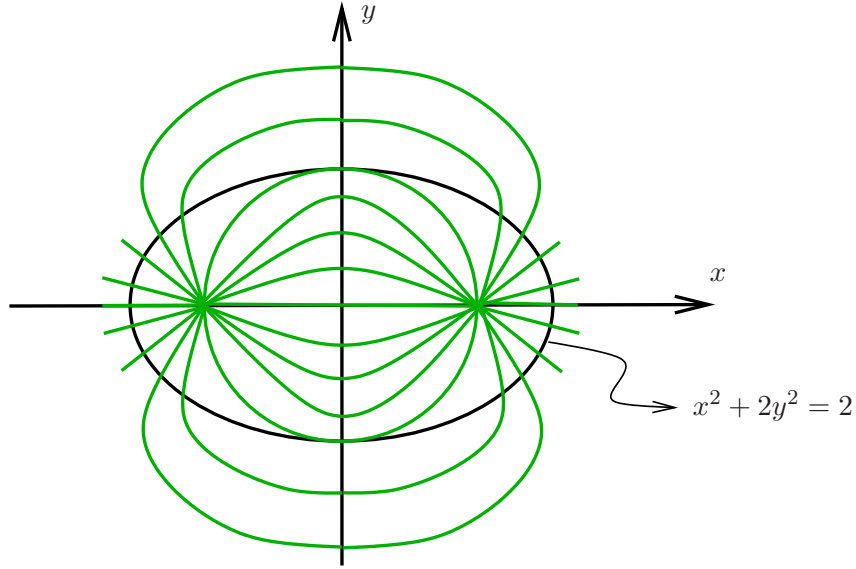}}}

\relabel{1}{$x$}

\relabel{2}{$y$}

\relabel{3}{$x^2 + 2y^2=2$}

\endrelabelbox
        \caption{The intersection of the ``pages" of $(\tilde{K}, \tilde{\theta})$ with the $xy$-plane, 
        and the ellipse $x^2 + 2y^2=2$.}
        \label{opbook}
\end{figure}

Let $E=\partial B$ denote the ellipsoid which is the boundary of the
domain:
   \[B=\{(x,y,z) \in \mathbb{R}^3 \;| \; x^2 + 2y^2 + z^2 \leq 2\}. \]
Note that, for all $t \neq \pm\pi/2$, the pages of the ``open
book'' $(\tilde{K}, \tilde{\theta})$  intersect $E$ transversely
inducing a foliation on $E \setminus \tilde{K}$, where $E \cap
\tilde{K} =\{(1,0,1), (1,0,-1), (-1, 0, -1), (-1, 0, 1)\}$. This
foliation agrees with what Gabai depicted in \cite[Fig.4]{G 83}. It
is invariant with respect to the reflections along all three
coordinate planes, and under a rotation of  angle $\pi$ about all
three coordinate axes.

The four points in $E \cap \tilde{K}$ are the corners of a square
inscribed in the circle of radius $\sqrt{2}$ on the $xz$-plane (see
Figure~\ref{ell}). Moreover, the map:
   \[\tilde{\rho}: E \to E, \; \; (x,y,z) \to (z, -y, -x) \]
 cyclically permutes these
four points, rotating the square (clockwise) in the $xz$-plane by an
angle of $\pi/2$. Furthermore,   $\tilde{\rho}$ is an orientation
reversing self-diffeomorphism of $E$ such that:
     \[\tilde{\theta} \circ \tilde{\rho} (x,y,z) = \tilde{\theta} (x,y,z) + \pi \;\;\; \mbox{for
          any} \;(x,y,z) \in E \setminus \tilde{K} . \]

\begin{figure}[ht]
  \relabelbox \small {
  \centerline{\epsfbox{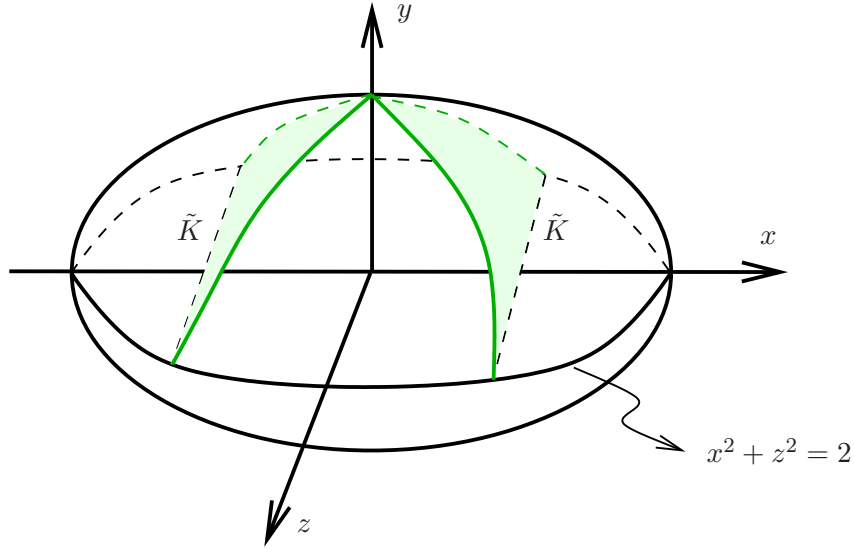}}}

\relabel{1}{$x$}

\relabel{2}{$y$}

\relabel{3}{$z$}

\relabel{4}{$x^2 + z^2=2$}

\relabel{5}{$\tilde{K}$}

\relabel{6}{$\tilde{K}$}

\endrelabelbox
        \caption{The intersection of the $(\pi / 2)$-page of $(\tilde{K}, \tilde{\theta})$  with  $B$}
        \label{ell}
\end{figure}

Let $M_i$ be an arbitrary closed oriented $3$-manifold for $i=1,2$,
and let $(K_i, \theta_i)$ be an open book in $M_i$. Our goal is to
construct an open book $(K, \theta)$ in the connected sum of $M_1$
and $M_2$ such that the page of $(K, \theta)$ is obtained by
plumbing the pages of $(K_1, \theta_1)$ and $(K_2, \theta_2)$.
Suppose that $C_i$ is a properly embedded arc in the page
$\theta_i^{-1} (0) \cup K_i$.  Then $C_i$ has a neighborhood $U_i
\subset M_i$ with an orientation-preserving diffeomorphism $\psi_i:
U_i \to \mathbb{R}^3$, carrying $(K_i \cap U_i , \theta_i|_{U_i})$
to $(\tilde{K}, \tilde{\theta})$ and $C_i$ to the segment $[-1, 1]$
on the $x$-axis.  This last claim follows from two basic facts:

(i)  any locally trivial fibration is trivial over an interval;

(ii)   the geometric monodromy can be assumed to be the identity near the
binding of an open book.

Consequently, the composition:
    \[\rho= \psi_2^{-1} \circ \tilde{\rho} \circ \psi_1: E_1=\psi_1^{-1}
            (E) \to E_2=\psi_2^{-1} (E) \]
is an orientation-reversing
diffeomorphism which can be used to construct the connected sum:
     \[M=    M_1 \# M_2= (M_1 \setminus \mbox{int} (B_1)) \cup_{\rho} (M_2
           \setminus \mbox{int}( B_2)),  \]
 where $B_i = \psi_i^{-1} (B)$.

There is a natural open book $(K, \theta)$ on $M$ which is defined
as follows: Let $K$ be the union of $K_1 \setminus \mbox{int} (B_1)$
and $K_2 \setminus \mbox{int} (B_2)$, which is a link in $M$ because
of the properties of the map $\tilde{\rho}$ discussed above.
Since $\theta_2 \circ \rho (x,y,z) = \theta_1 (x,y,z) + \pi$,
the map $\theta$ defined as $\theta_i +(-1)^{i+1} \pi / 2$ when
restricted to $(M_i \setminus \mbox{int} (B_i)) \setminus K_i $
induces a fibration on $M \setminus K$.

\begin{figure}[ht]
  \relabelbox \small {
  \centerline{\epsfbox{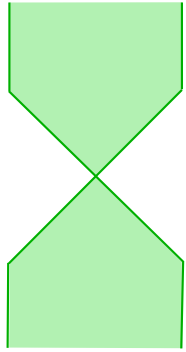}}}

\endrelabelbox
        \caption{$(\tilde{\theta}^{-1}(\pi / 2) \cup \tilde{K}) \setminus \mbox{int}( B)$}
        \label{page}
\end{figure}

To understand the pages of the open book $(K, \theta)$  on
$M$, consider the piece of (non-smooth) surface
$(\tilde{\theta}^{-1}(\pi / 2) \cup \tilde{K}) \setminus
\mbox{int}( B)$ depicted in Figure~\ref{page} (compare with
Figure~\ref{ell}, but beware that we take the complement). Since
$(\tilde{K}, \tilde{\theta})$ is a local model for both open books
$(K_1, \theta_1)$ and  $(K_2, \theta_2)$, we just need to understand
how the pages in two copies of this local model fit together by the
map $\tilde{\rho}: E \to E$. Because of the symmetry of the
construction, $(\tilde{\theta}^{-1}(-\pi / 2) \cup \tilde{K})
\setminus \mbox{int}( B)$ is also homeomorphic to the surface
depicted in Figure~\ref{page}.   These two \emph{oriented}
surfaces-with-boundary can be glued together along parts of their
boundaries, dictated by the map $\tilde{\rho}: E \to E$, to give an
oriented smooth surface with corners as we depicted on the left in
Figure~\ref{plumpages}.

\begin{figure}[ht]
  \relabelbox \small {
  \centerline{\epsfbox{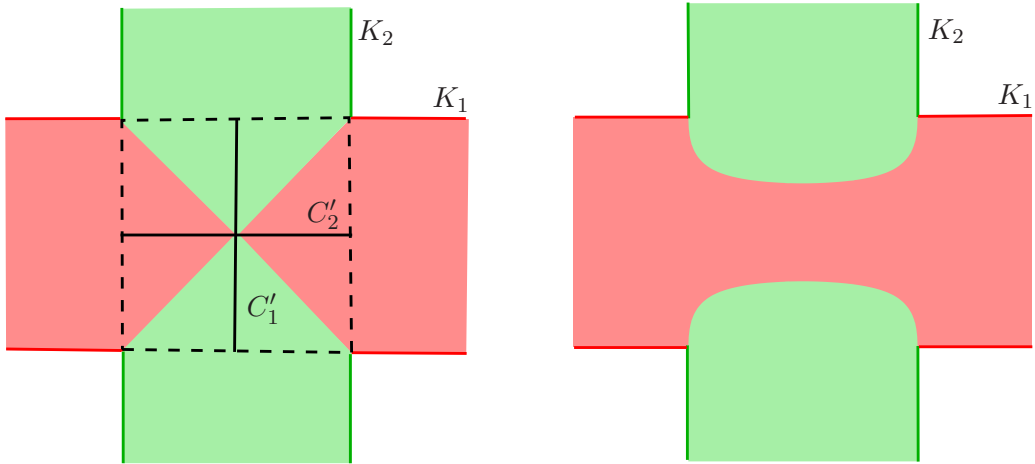}}}

\relabel{1}{$C'_2$} \relabel{2}{$C'_1$}

\relabel{3}{$K_2$} \relabel{4}{$K_1$}

\relabel{5}{$K_2$} \relabel{6}{$K_1$}

\endrelabelbox
        \caption{Local pictures of the pages of $(K, \theta)$: the $0$-page on the left,
            other pages on the  right}
        \label{plumpages}
\end{figure}

This shows that the $0$-page of $(K, \theta)$ can be viewed
 as the plumbing of the $(-\pi/2)$-page of $(K_1, \theta_1)$
 with the $(\pi/2)$-page of $(K_2, \theta_2)$ along the
 neighborhoods of the
 arcs $C'_1$ and $C'_2$ defined by:
       \[ C'_1 = \psi_1^{-1} (C_1) \;\mbox{and}\; C'_2 = \psi_2^{-1} (C_2),  \]
 where:
       \[ C_1= \{ x^2 +y^2  =1, \; y \leq 0, \;z=0\} \; \mbox{and} \; C_2=\{ x^2 +y^2
                    =1, \; y \geq 0, \;z=0\}.  \]

 Similarly, all the other pages of the open book $(K, \theta)$ will appear locally as drawn on the
 right in Figure~\ref{plumpages}, each of which is \emph{globally} diffeomorphic to the $0$-page,
 after smoothing the corners as usual. Hence $\theta: M \setminus K \to \mathbb{S}^1$
 is a locally trivial fibration each of whose fibers is obtained by plumbing a page
 of  $(M_1, \theta_1)$ with a page of  $(M_2, \theta_2)$---which finishes Gabai's proof of Stallings'
 Theorem~\ref{thm:stal}.

\medskip

The proof above can be described with another point of view which
turns out to be more suitable for the generalizations we have in
mind. One can interpret what is inside the domain $B$ in the local
model $(\mathbb{R}^3, \tilde{\theta})$  as the union of two
(overlapping) pieces:
\begin{itemize}
     \item a tubular neighborhood of the intersection $B
\cap \tilde{K}$, which is nothing but two disjoint arcs of the
binding $\tilde{K}$;
      \item a \emph{thickening} of the plumbing region.
 \end{itemize}

\begin{figure}[ht]
  \relabelbox \small {
  \centerline{\epsfbox{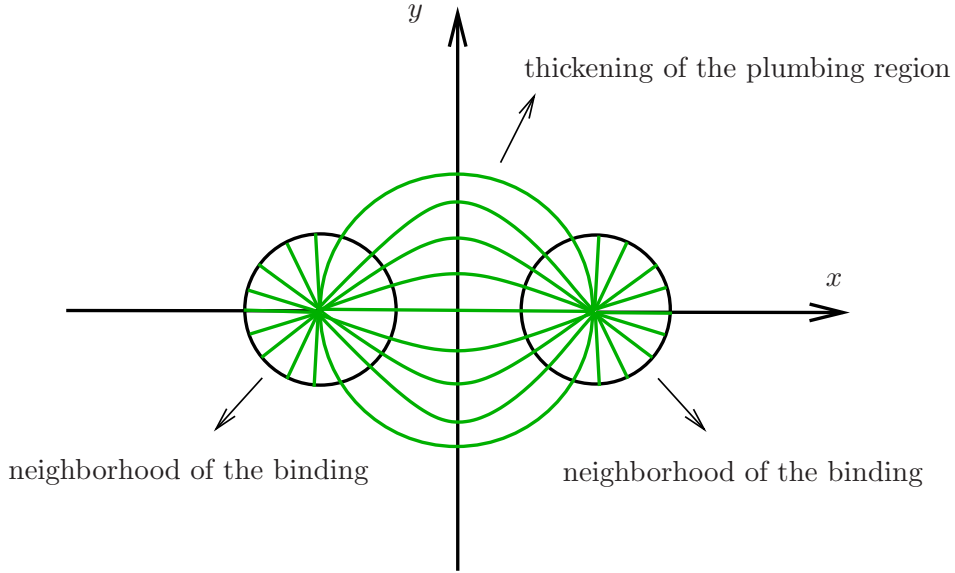}}}

\relabel{1}{$x$}

\relabel{2}{$y$}

\relabel{3}{thickening of the plumbing region}

\relabel{4}{neighborhood of the binding}

\relabel{5}{neighborhood of the binding}

\endrelabelbox
        \caption{A model with truncated pages}
        \label{opbo}
\end{figure}

The thickening (topologically a rectangle times an interval)
consists of a rectangle from each page $\tilde{\theta}^{-1}(t) \cup
\tilde{K}$ for $t \in [-\pi/2, \pi/2] \in \mathbb{S}^1$. To see
this, we slightly truncate the pages of $\tilde{\theta}$ in $B$
corresponding to the arc $[-\pi, -\pi/2] \cup [\pi/2, \pi]$ on
$\mathbb{S}^1$ such that the pages intersect the $xy$-plane as shown
in Figure~\ref{opbo}. In other words, we slightly deform the domain
$B$ keeping all of its symmetries needed in the previous discussion.
Therefore, by removing $B$, we remove the plumbing region from half
of the pages of the open book corresponding to one ``half" of
$\mathbb{S}^1$, along with tubular neighborhood of the two arcs of
the binding.

For the Murasugi sum of two open books, we remove the plumbing
regions from half of the pages in both open books \emph{but these
halves correspond to complementary oriented arcs} on $\mathbb{S}^1$.
(This fact reveals itself in the above proof by the appearance of
the difference $\pi$ in the parametrization of the fibrations to be
glued.) So, when we glue the ambient manifolds after removing
diffeomorphic copies of $B$ from each one of them,  the fibrations
in the complements of the respective bindings will glue together so
that the hole created as a result of removing a rectangle (the
plumbing region)  from any page will be sewn back up by the
rectangle in the  corresponding  page of the complementary
fibration. The way that these rectangles are identified is
equivalent to plumbing, so that the resulting pages are smooth
manifolds. One can also see that the aforementioned tubular
neighborhoods of the arcs on the bindings will indeed disappear in
the process, whereas the rest of the bindings will glue together to
give the new binding in the glued up manifold.

\medskip

There is yet another interpretation of the proof using abstract open
books (see Etnyre \cite[Theorem 2.17]{E 06}).  Given two abstract
open books $(\Sigma_i, \phi_i)$, $i = 1, 2$  (see Remark
\ref{rem:eqfib} (\ref{rem:aob}) below), let $C_i$ be an arc properly
embedded in $\Sigma_i$ and $R_i = C_i \times [-1, 1] \subset
\Sigma_i$ a rectangular neighborhood of $C_i$. The idea of the proof
is to perform a Murasugi sum of the mapping tori $\cM(\Sigma_1,
\phi_1)$ and $\cM(\Sigma_2, \phi_2)$ leaving the bindings out of the
picture at first and then to complete the resulting mapping torus
into an open book of the connected sum of the ambient manifolds.

\begin{figure}[ht]
  \relabelbox \small {
  \centerline{\epsfbox{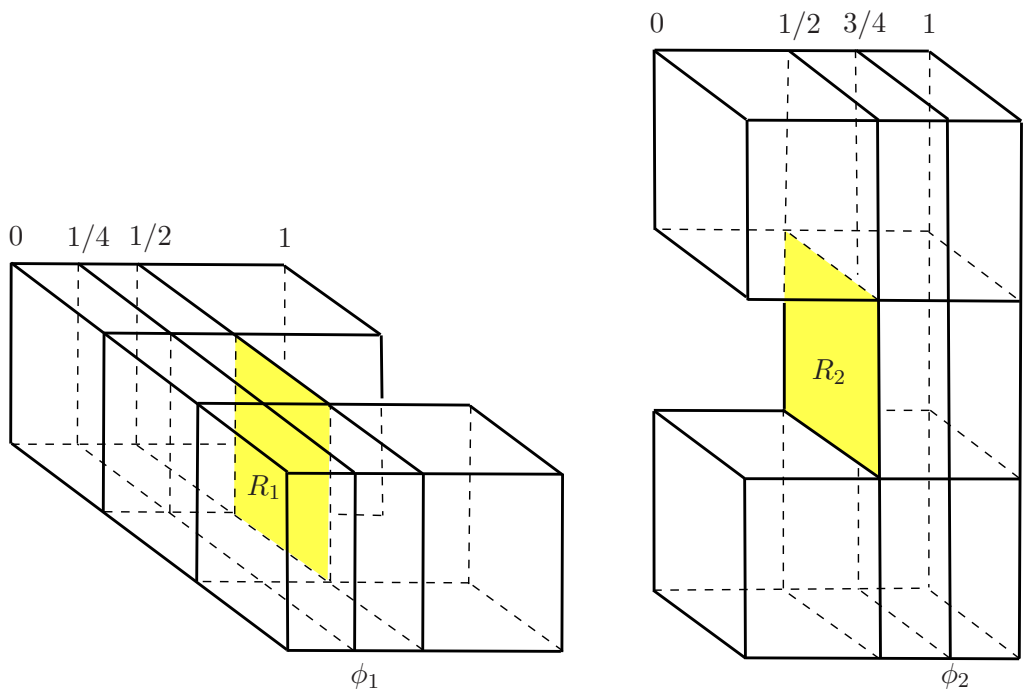}}}
\relabel{1}{$\phi_1$}

\relabel{2}{$\phi_2$}

\relabel{3}{$0$} \relabel{4}{$1/4$} \relabel{5}{$1$}
\relabel{6}{$0$} \relabel{7}{$3/4$} \relabel{8}{$1$}

\relabel{a}{$R_1$}

\relabel{b}{$R_2$}

\relabel{c}{$1/2$}

\relabel{d}{$1/2$}
\endrelabelbox
        \caption{Local pictures of $\cM(\Sigma_1, \phi_1) \setminus B_1$ (on the left) and
         $\cM(\Sigma_2,  \phi_2)\setminus B_2$ (on the right) }
        \label{W}
\end{figure}

Note that $B_1 = R_1 \times [ 1/ 2 , 1]$ is a 3-ball in $\cM(\Sigma_1,
\phi_1)$ and similarly $B_2 = R_2 \times [0, 1/ 2 ]$ is a $3$-ball
in $\cM(\Sigma_2, \phi_2)$. We view the mapping torus $\cM(\Sigma_1,
\phi_1)$ obtained as gluing $\Sigma_1 \times\{0\}$ to
$\Sigma_1\times\{1\}$ using the identity and then cutting the
resulting $\Sigma_1 \times \mathbb{S}^1$ along $\Sigma_1
\times\{1/4\}$ and regluing using $\phi_1$ (see
Figure~\ref{W}). Similarly we view $\cM(\Sigma_2, \phi_2)$ obtained
as gluing $\Sigma_2 \times\{0\}$ to $\Sigma_2\times\{1\}$ using the
identity and then cutting the resulting $\Sigma_2 \times
\mathbb{S}^1$ along $\Sigma_2 \times\{3/4\}$ and regluing using
$\phi_2$.

\begin{figure}[ht]
  \relabelbox \small {
  \centerline{\epsfbox{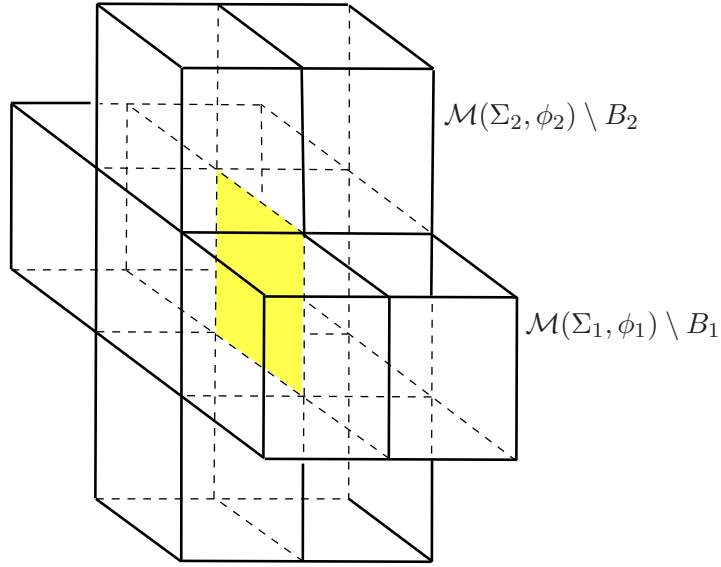}}}

\relabel{1}{$\cM(\Sigma_1, \phi_1) \setminus B_1$}

\relabel{2}{$\cM(\Sigma_2,    \phi_2)\setminus B_2$}

\endrelabelbox
        \caption{ Two ``lego" pieces of Figure~\ref{W} fitting together}
        \label{V}
\end{figure}

Let $\Sigma = \Sigma_1 + \Sigma_2$ denote the Murasugi sum of
$\Sigma_1$ and $\Sigma_2$ along the rectangles $R_1$ and $R_2$. Then
$\cM(\Sigma_1, \phi_1) \setminus B_1$ and $\cM(\Sigma_2, \phi_2)
\setminus B_2$ can be glued together, as illustrated in
Figure~\ref{V},  to induce a mapping torus with page $\Sigma$.
Therefore we conclude that:
     \[   \big(\cM(\Sigma_1, \phi_1) \setminus B_1 \big)\cup
          \big(\cM(\Sigma_2, \phi_2) \setminus B_2\big)= \cM(\Sigma, \phi), \]
where $\Sigma= \Sigma_1 + \Sigma_2$, and $\phi= \phi_1 \circ
\phi_2$. Here we extend $\phi_i$ ($i=1,2$) from $\Sigma_i$ to
$\Sigma$ by the identity map and then compose these extended
diffeomorphisms, which we  still denote by $\phi_i$ on $\Sigma$.  As
a matter of fact, from this monodromical viewpoint ``Murasugi sum"
appears more like a \emph{composition} than a sum.

To show that the mapping torus $\cM(\Sigma, \phi)$ extends to an
open book of the connected sum $M_1 \# M_2$ we proceed as follows
(see Goodman's Thesis \cite[pages 9-10]{Go 03}). First of all, we view
$\Sigma_i$ as a submanifold of $\Sigma$ and identify $R=R_i$, for
$i=1,2$. Then $s_i=:R \cap
\partial \Sigma_i$ is the disjoint union of two properly embedded
arcs in $\Sigma$ such that the set of four points  $\partial
s_1=\partial s_2$ belongs to $\partial \Sigma$.

In the following we present the separating sphere $S$ in the
connected sum $M=M_1 \# M_2$. Let $I_1=[0,1/2]$ and $I_2=[1/2, 1]$.
For each $i=1,2$, consider the disjoint union of two disks $s_i \times I_i  \subset
\Sigma \times I \subset \cM(\Sigma, \phi)$. Let $S'$ be the surface
obtained as the following union of six disks:
     \[ (s_1 \times I_1) \cup (s_2 \times I_2) \cup (R \times \{0\}) \cup (R \times \{1/2\}) \]
in $\cM(\Sigma, \phi)$. Observe that $\partial S' = \bS^1 \times \partial s_1 $. We
can cap off $S'$ with the four disks $ \bD^2 \times \partial s_1 $
to construct the desired sphere $S$ as illustrated in
Figure~\ref{T}.

\begin{figure}[ht]
  \relabelbox \small {
  \centerline{\epsfbox{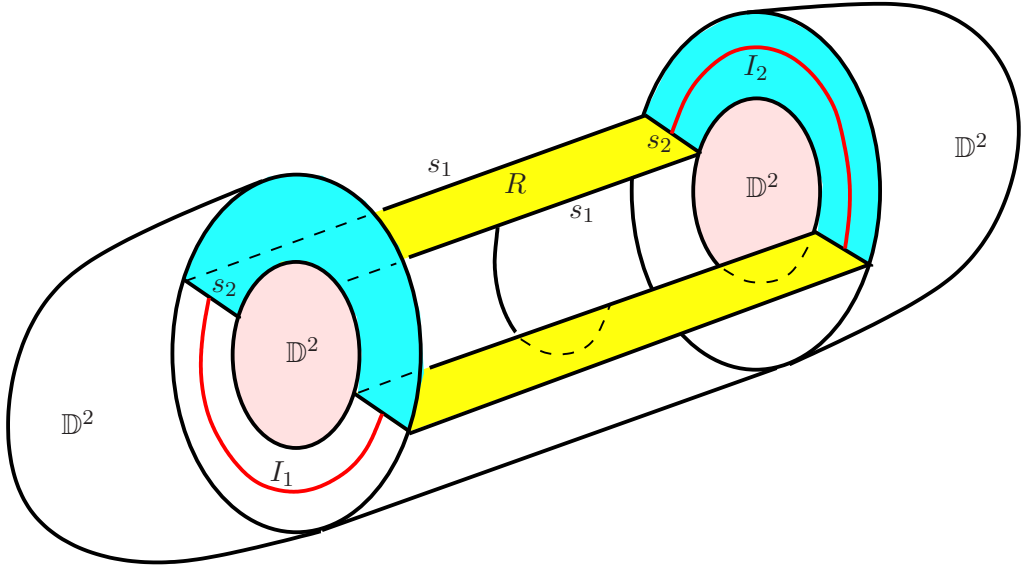}}}

\relabel{1}{$R $}

\relabel{6}{$s_2$}

\relabel{2}{$I_2$}

\relabel{3}{$s_2$}

\relabel{4}{$s_1$} \relabel{7}{$s_1$}

\relabel{5}{$I_1$}

\relabel{a}{$\bD^2$}

\relabel{b}{$\bD^2$}

\relabel{c}{$\bD^2$}

\relabel{d}{$\bD^2$}

\endrelabelbox
        \caption{The four disks used to cap off $S'$ in order to get the sphere $S$}
        \label{T}
\end{figure}

Now we claim that $M \setminus S$ decomposes into $M_1
\setminus B_1$ and $M_2 \setminus B_2$ for some $3$-dimensional
balls $B_1$ and $B_2$. To prove our claim, we note that $\cM(\Sigma,
\phi)= (\Sigma \times I_1) \cup (\Sigma \times I_2)$, where we
identify $\Sigma \times \{1/2\}$ in the first product with $\Sigma
\times \{1/2\}$ in the second product via $\phi_1$ and $\Sigma
\times \{1\}$ with $\Sigma \times \{0\}$ via $\phi_2$. It follows
that, by removing $S$, we have $(\Sigma_1 \sqcup (\Sigma_2 \setminus
R))\times I_1$ glued to $(\Sigma_2 \sqcup (\Sigma_1 \setminus
R))\times I_2$. But since $\phi_1$ is the identity on $\Sigma_2$ and
$\phi_2$ is the identity on $\Sigma_2$, the result can also be
viewed as a union of two pieces  $\cM(\Sigma_1, \phi_1) \setminus (R
\times I_1)$ and $\cM(\Sigma_2, \phi_2) \setminus (R \times I_2)$.

Finally, we insert the binding as follows. Since $\partial
s_i$ is a set of four points in $\partial \Sigma$, the solid torus
$\bD^2 \times \partial \Sigma$ is cut into four pieces along $\bD^2
\times \partial s_i$. Thus by gluing in the binding, we see that $M$
decomposes into two pieces along the sphere $S$:
     \[  \big(\cM(\Sigma_1,
\phi_1)\cup (\bD^2 \times \partial \Sigma_1)\big) \setminus \big((R
\times I_1) \cup (\bD^2 \times s_1)\big)= M_1 \setminus \big((R
\times I_1) \cup (\bD^2 \times s_1)\big)\]
and:
    \[ \big(\cM(\Sigma_2, \phi_2) \cup (\bD^2 \times
\partial \Sigma_2)\big) \setminus \big((R \times I_2) \cup (\bD^2 \times s_2)\big)= M_2
\setminus \big((R \times I_2) \cup (\bD^2 \times s_2)\big). \]
Observe that each $B_i:=(R \times I_i) \cup (\bD^2 \times s_i)$ is a
$3$-ball with boundary $S$.

\medskip
Our paper is motivated by the search of the most general operation
of Murasugi-type sum (that is, embedded Milnor-style plumbing) for
which one has an analog of Theorem \ref{thm:stal}. We figured out
that we do not need to restrict in any way the full-dimensional
submanifolds which are to be identified in the plumbing operation.
That is why we define a general operation of ``summing'' of
manifolds (see Definition \ref{def:absum}), which reduces to the
classical operation of Definition \ref{def:plumb} when the
identified submanifolds have product structures $\bD^n \times \bD^n$.

The greater level of generality obliged us to discard the
special model used in the previous proof. The principle of the proof
of our generalization \ref{thm:genstal} of Gabai's theorem is
instead inspired by Etnyre's interpretation. In this respect, Figure
\ref{V} is to be compared with Figure \ref{O}.

\bigskip
\section{Conventions and basic definitions}
\label{sec:notconv}

In this section we explain our conventions about manifolds,
orientations and coorientations of hypersurfaces. We give rather
detailed explanations because throughout the paper we work without
any assumptions about orientability of the manifolds: the
only important issues are about \emph{coorientations}, which makes
the setting rather non-standard when compared with the usual
literature in differential topology.
\medskip

In this paper, the  manifolds are assumed to be smooth and
\emph{pure dimensional}, but not necessarily orientable or
connected. If a manifold is endowed with an orientation, we
explicitly say that it is an ``\emph{oriented manifold}''. We use
the expression ``\emph{manifold-with-boundary}'' for a smooth
manifold with possibly empty boundary. We denote by $\partial W$ the
boundary of the manifold-with-boundary $W$ and by:
   \[  \mbox{int}(W) := W \setminus     \partial W  \]
its \emph{interior}.

In the sequel, we will implicitly use the facts that \emph{the
corners of a manifold with corners can be smoothed}, and that the
resulting smooth manifold-with-boundary is well-defined up to
isotopy as a zero-codimensional submanifold of the initial manifold
with corners. A standard reference for these folklore facts is the
Appendix of Milnor's paper  \cite{M 59}. We will also use the
folklore fact that \emph{two manifolds-with-boundary can be glued
along compact zero-codimensional submanifolds of their boundaries,
once a diffeomorphism between these submanifolds is fixed}, and that
the result is well-defined up to diffeomorphism. A standard
reference for this is Hirsch's book \cite[Chapter 8.2]{H 76}. All
these facts are also treated in a detailed way by Douady in his
contributions \cite{D 62-1}, \cite{D 62-2}, \cite{D 62-3} to the
Seminar Cartan.

\begin{remark}  \label{rem:simplnot}
    In the sequel, the only gluings to be done will be special cases of identifications
of submanifolds of two manifolds-with-boundary by diffeomorphisms.
In order to simplify the notations, instead of giving different names
to those submanifolds and labeling also the diffeomorphism used for
the gluing, we will assume that the two submanifolds were identified
using that diffeomorphism, which implies that the gluing
diffeomorphism is the identity. For instance, we will not write
``glue $M_1$ to $M_2$ using the diffeomorphism $\phi: P_1 \to P_2$
of $P_i \hookrightarrow M_i$'', but ``glue $M_1$ to $M_2$ along $P
\hookrightarrow M_i$''.
\end{remark}

If $V$ is a submanifold-with-boundary embedded in $W$, then we
use the notation $V \hookrightarrow W$.
 We say that $V$ is {\bf properly embedded} into $W$ if $V \cap \partial W =
 \partial V$ and if $V$ and $\partial W$ are transversal in $W$ everywhere
 along $\partial V$. When $\partial V = \emptyset$, this means simply that
 $V \subset \mbox{int}(W)$. In this paper,  the submanifolds of interest
 are not necessarily properly embedded (for instance, the
pages of an arbitrary open book). If $M \hookrightarrow W$ is a submanifold, we denote by
$\codim_W(M)$ its \emph{codimension} in the ambient manifold $W$.

If $V  \hookrightarrow W$ is properly embedded, we denote by
    $\cU_W(V)$ (or simply $\cU(V)$ if $W$ is clear from the context)
    a closed {\bf tubular neighborhood} of $V$ in $W$ such that
    $\cU_W(V) \cap \partial W$
    is a tubular neighborhood of $\partial V$ in $\partial W$.
    Moreover,  we assume that $\cU_W(V)$ is endowed
    with a structure of smooth fiber bundle over $V$,
    whose fibers are diffeomorphic to compact balls
    of dimension $\codim_W(V)$.

Let us examine the special case of properly embedded
\emph{hypersurfaces}. One has the following well-known proposition:

\begin{proposition}
   Let $M \hookrightarrow W$ be a compact hypersurface-with-boundary
   properly embedded inside the manifold $W$.
  The following conditions are equivalent:
    \begin{enumerate}
        \item  \label{condnorm} the normal bundle $\cN_{M | W}$ of $M$ in $W$ is orientable;
        \item  \label{condprod} $M$ admits a tubular neighborhood diffeomorphic to
            $[-1, 1] \times M$, where $M \hookrightarrow W$ is identified
            with $\{0\} \times M$;
        \item  \label{condreg} each connected component $U_i$ of an arbitrary regular
           neighborhood $\cU_W(M)$ is disconnected by $U_i \cap M$.
    \end{enumerate}
  Moreover, if any of  the conditions above is satisfied, then the following choices are equivalent:
       \begin{enumerate}
        \item[(1')]  an orientation of the normal bundle $\cN_{M | W}$ of $M$ in $W$;
        \item[(2')]  an embedding
            $[-1, 1] \times M \hookrightarrow W$ which sends  $\{0\} \times M$
            to $M$ by $\{0\} \times m \to m$ for any $m \in M$, up to isotopy;
        \item[(3')]  a choice of connected component of $U_i \setminus M$
          for each connected component $U_i$ of a tubular neighborhood
           $\cU_W(M)$.
    \end{enumerate}
    More precisely, the normal vectors pointing towards the positive side for the
    chosen orientation of the normal bundle are tangent to the curves entering into
    $(0, 1] \times M$, which defines the choice of connected component
    of each $U_i$.
\end{proposition}

The previous proposition allows us to define:

  \begin{definition}  \label{def:coorient}
 Let $M \hookrightarrow W$ be a properly embedded compact hypersurface-with-boundary.
 It is called {\bf coorientable} if it satisfies any one of the
 equivalent conditions (\ref{condnorm})--(\ref{condreg}) of the previous
 proposition.
  A {\bf coorientation} of $M$ in $W$ is an orientation of the normal
  bundle $\mathcal{N}_{M | W}$ of $M$ in $W$.
    \end{definition}

\begin{example}  Consider a M\" obius band $W$ seen as a non-trivial
    segment-bundle over a circle. Any fiber is coorientable, but no
  section of it is coorientable. \end{example}

Suppose that $W$ is a manifold with nonempty boundary $\partial W$.
Recall that we do not assume orientability of either $W$ or
$\partial W$. Even though $\partial W$ is not properly embedded in
$W$, it has an orientable normal bundle in $W$ and hence we say that
$\partial W$ is coorientable by adapting
Definition~\ref{def:coorient} to this case. Since $\partial W$
is coorientable, then any codimension zero submanifold of $\partial
W$ is coorientable and for each connected component of such a
submanifold of $\partial W$, the two coorientations may be
distinguished as:
  \begin{itemize}
      \item {\bf incoming}, if the corresponding normal vectors point inside $W$;
      \item {\bf outgoing}, if the corresponding normal vectors point outside $W$.
  \end{itemize}

\begin{remark} \label{rem:strictincl}
     In the sequel (see for instance
Definition~\ref{def:cobcorners}) we will not necessarily coorient a
whole boundary component uniformly, but we might have to break it up
by inserting ``corners" as in Figure~\ref{B}. For this reason, we
also  speak about the coorientation of any
full-dimensional submanifold of the boundary.
\end{remark}

If a  manifold-with-boundary $W$ is \emph{oriented}, then for
each connected component of its boundary $\partial W$ we define the
{\bf outgoing orientation} by the rule known as ``outside pointing
normal vector comes first'': a normal vector to $\partial W$
pointing outside of $W$, followed by a positive basis of the tangent
space to $\partial W$, gives a positive basis to the tangent space
of $W$. It is customary to take the outgoing orientation as
the \emph{canonical orientation} induced on $\partial W$. The opposite
orientation of the boundary is the {\bf incoming orientation}.

\begin{example} For each $n \geq 1$, we denote by $\bD^n \subset \R^n$ the compact
unit ball endowed with the restriction of the canonical orientation
of $\R^n$ and by $\bS^{n-1}$ its boundary sphere, endowed with the
associated outgoing orientation.  \end{example}

If $W$ is an oriented manifold-with-boundary and $\partial W$ is
\emph{independently} oriented, then:
 \begin{itemize}
    \item  its {\bf outgoing boundary} $\partial_+ W$ is the union of the connected
        components of $ \partial W  $ which are endowed with the outgoing
        orientation;
    \item   its {\bf incoming boundary} $\partial_- W$ is the union of the connected
        components of $ \partial W $ which are endowed with the incoming orientation.
 \end{itemize}

 We clearly have:
   \[ \partial W = \partial_+  W    \bigsqcup  \:  \partial_- W.  \]

In this case, we see $W$ as a {\bf cobordism from $\partial_- W$ to
$\partial_+ W$} (see Figure \ref{A}).

\begin{figure}[ht]
  \relabelbox \small {
  \centerline{\epsfbox{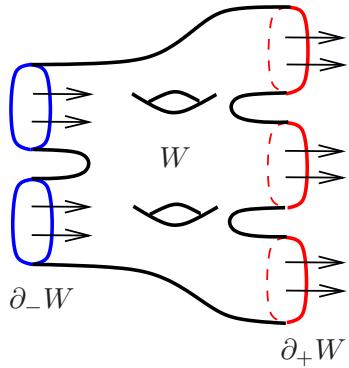}}}

\relabel{1}{$\partial_- W$}

\relabel{2}{$\partial_+ W$}

\relabel{3}{$W$}

\endrelabelbox
        \caption{$W$ is a cobordism from $\partial_- W$ to $\partial_+ W$.}
        \label{A}
\end{figure}

For instance, the interval $[0, 1]$ endowed with its canonical
orientation is a cobordism from the positively oriented point
$\{0\}= \partial_- [0, 1]$ to the positively oriented point $\{1\} =
\partial_+  [0, 1]$. Note that to orient a point means to
choose one of the signs $\pm$ attached to it, which allows us to
speak in this case about positive/negative orientations.

More generally, we will denote by $I$ or $I_j$ ($j$ varying inside
some index set) an {\bf oriented compact interval}, that is, an
oriented compact manifold-with-boundary, diffeomorphic to $[0, 1]
\subset \R$. Its two boundary components will be endowed with their
canonical orientations, therefore we may speak without ambiguity of
the {\bf outgoing point} $\partial_+ I$ and the {\bf incoming point}
$\partial_- I$ of $I$.

\begin{definition}  \label{def:sides}
   Let $M \hookrightarrow W$ be a \emph{properly embedded} and \emph{cooriented}
   compact hyper\-surface-with-boundary.
   A {\bf positive side} of $M$
   is an embedding $I^+ \times M \hookrightarrow W$ such that
   $M \hookrightarrow W$ is identified with
   $\partial_- I^+ \times M$ and the positive normal vectors of $M$
   point into $I^+ \times M$.
   A {\bf negative side} of $M$
   is an embedding $I^- \times M \hookrightarrow W$ such that
   $M \hookrightarrow W$ is identified with
   $\partial_+ I^- \times M$ and the positive normal vectors of $M$
   point outside it. Here both $I^+$ and $I^-$ denote oriented compact intervals.
   A {\bf collar neighborhood} of $M \hookrightarrow W$ is the union of a negative and of a
   positive side of $M$ whose intersection is $M$.
 \end{definition}

In the sequel we will have to work with a more general notion of
cobordism, which is described in the next section.

\bigskip
\section{Cobordisms of manifolds-with-boundary}
\label{sec:cobcorn}

In this section we set up the notation for \emph{cobordisms of
manifolds with boundary}, without the assumption of orientability.
We also introduce \emph{cylinders},
 \emph{cylindrical cobordisms} and \emph{endobordisms} as particular cases of
cobordisms of manifolds with boundary. Moreover, we explain in which
sense the notions
 of endobordism and properly embedded cooriented hypersurface in a manifold-with-boundary
 are equivalent.
 \medskip

In the next definition we extend the notion of cobordism to
situations where:
 \begin{itemize}
    \item  the total manifold is not necessarily orientable;
    \item the incoming and outgoing boundaries are not necessarily closed manifolds;
    \item the total manifold may have boundary components which are not labeled
       as incoming or outgoing.
 \end{itemize}

 What we keep instead from the situation described in the previous section
    is the disjointness of the
two types of boundary regions and the fact that they are of
codimension zero in the boundary of the cobordism.

\begin{definition}  \label{def:cobcorners}
    Let $M^-$ and $M^+$
be manifolds-with-boundary. A {\bf cobordism} $W$ from $M^-$ to
$M^+$ is a manifold-with-boundary $W$, whose boundary is decomposed
as:
     $$\partial W = Y \cup M^- \cup M^+,$$
     where
$Y$ is a nonempty submanifold-with-boundary of $\partial W$ such
that $M^- \cap M^+ =\emptyset $, $Y \cap M^-=
\partial M^-$, $Y \cap M^+=
\partial M^+$ and:
       \begin{itemize}
             \item  $M^-$ is endowed with
               the incoming coorientation, and
            \item  $M^+$ is endowed with
               the outgoing coorientation.
       \end{itemize}
    We say that  $M^{\mp}$ is the {\bf incoming/outgoing boundary region}
      of the cobordism $W$ and set $\partial_{\mp}
    W = M^{\mp}$. We denote this cobordism (of manifolds-with-boundary) by:
         $$W : \partial_- W  \Longmapsto  \partial_+ W.$$

\begin{figure}[ht]
  \relabelbox \small {
  \centerline{\epsfbox{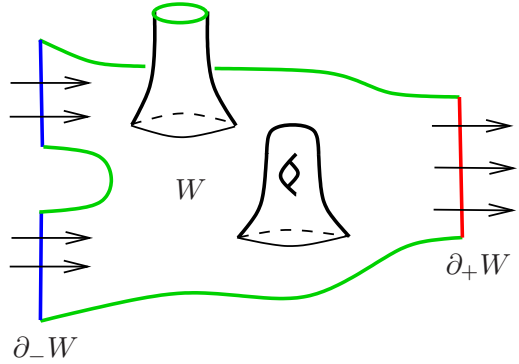}}}

\relabel{1}{$\partial_- W$}

\relabel{2}{$\partial_+ W$}

\relabel{3}{$W$}

\relabel{4}{$Y$}

\endrelabelbox
        \caption{Cobordism of
manifolds-with-boundary $W : \partial_- W  \Longmapsto  \partial_+
W$, where
    $\partial_- W$ is blue, $\partial_+ W$ is red and $Y$ (the rest of
    $\partial W$) is green (colour figure online).}
        \label{B}
\end{figure}

\begin{remark} \label{rem:sameb}
     The definition above is not new
     (see, for example, \cite{BNR 12})
        except for the orientability
assumptions. Strictly speaking, $W$ is a manifold with corners (for
this reason, we called them ``\emph{cobordisms with corners}'' in a
previous version of this paper), but nevertheless, corners along
$\partial (\partial_- W) \:  \sqcup \:
\partial (\partial_+ W)$ may be smoothed. Note that $\partial_- W$
and $\partial_+ W$ may belong to the same connected component of
$\partial W$ after smoothing the corners and also, the boundary of
$W$ may have connected components disjoint from $\partial_- W \sqcup
\partial_+ W$, as illustrated  in Figure \ref{B}.
     \end{remark}

     More generally,
    if one has two manifolds $M^-$ and
    $M^+$ (possibly with boundaries) and fixed diffeomorphisms $M^\pm \simeq \partial _\pm W$,
    we simply say that $W$ is a {\bf cobordism from $M^-$ to $M^+$}
    and  write $W : M^- \Longmapsto M^+$. Note that cobordisms
     can be \emph{composed}: if $W_1 : M_1 \Longmapsto M_2$ and
$W_2 : M_2 \Longmapsto M_3$ are two cobordisms  then their {\bf
composition} $W_2 \circ W_1: M_1 \Longmapsto M_3$ is a cobordism
obtained by gluing $W_1$ and $W_2$ along $M_2$.
\end{definition}

\begin{remark}
   The notion of cobordism of manifolds-with-boundary weakens and
   extends to arbitrary dimension the notion of
   \emph{sutured manifold} introduced in dimension $3$ by Gabai \cite[Definition 2.6]{G 83bis}:
       \begin{quote}
           ``A sutured manifold $(M, \gamma)$ is a compact oriented $3$-manifold
           $M$ together with a set $\gamma \subset \partial M$ of pairwise disjoint
           annuli $A(\gamma)$ and tori $T(\gamma)$. Furthermore, the interior of each
           component of $A(\gamma)$ contains a \emph{suture}, i.e. a homologically
           nontrivial oriented simple closed curve. We denote the set of sutures by
           $s(\gamma)$. Finally every component of $R(\gamma) = \partial M -
           \mbox{int}(\gamma)$
           is oriented. Define $R_+(\gamma)$ (or $R_-(\gamma)$) to be those
           components of $\partial M - \mbox{int}(\gamma)$ whose normal vectors point out
           of (into) $M$. The orientations on $R(\gamma)$ must be coherent with
           respect to $s(\gamma)$, i.e., if $\delta$ is a component of $\partial R(\gamma)$
           and is given the boundary orientation, then $\delta$ must represent the same
           homology class in $H_1(\gamma)$ as some suture.''
       \end{quote}
    A sutured manifold $(M, \gamma)$ as in Gabai's definition is a
    cobordism of manifolds-with-boundary from $R_-(\gamma)$ to
 $R_+(\gamma)$ according
    to our definition. We drop any constraints on the structure of
    the complement of the union of outgoing and incoming boundary regions inside
    the full boundary. Moreover, we do not assume that the ambient manifold is oriented, or even
    orientable. Our definition is also more general than the extension to arbitrary dimensions
    of the notion of sutured manifold, given by Colin-Ghiggini-Honda-Hutchings in
    \cite{CGHH 11}.
\end{remark}

\emph{In the sequel, we will simply write  ``cobordisms'' instead of
``cobordisms of manifolds-with-boundary''.}

\begin{definition} \label{def:maptor}
   If the incoming and the outgoing boundaries $M^-$ and $M^+$ of a cobordism
   $W : M^- \Longmapsto M^+$ are diffeomorphic and a diffeomorphism between
   them is fixed, then
   we say that $W$ is an {\bf endobordism} of $M \simeq M^- \simeq M^+$.
   The {\bf mapping torus} of the endobordism
   $W : M^- \Longmapsto M^+$ is the manifold-with-boundary $T(W)$
   obtained by gluing $M^-$ and $M^+$ using this diffeomorphism. The mapping torus comes
    equipped with a cooriented proper embedding
    $M \hookrightarrow T(W)$, which is the image inside $T(W)$ of the boundary
    manifolds $M^-$ and $M^+$
    which are identified (see Figure \ref{C}).
\end{definition}

In the notation ``$T(W)$'', we suppress for simplicity the
diffeomorphism which identifies the incoming and outgoing
boundaries. Note that it is nevertheless of fundamental importance
for the construction. The reason we chose the name ``mapping torus''
is explained in Remark \ref{rem:eqfib} (\ref{rem:classic}) below.

\begin{figure}[ht]
  \relabelbox \small {
  \centerline{\epsfbox{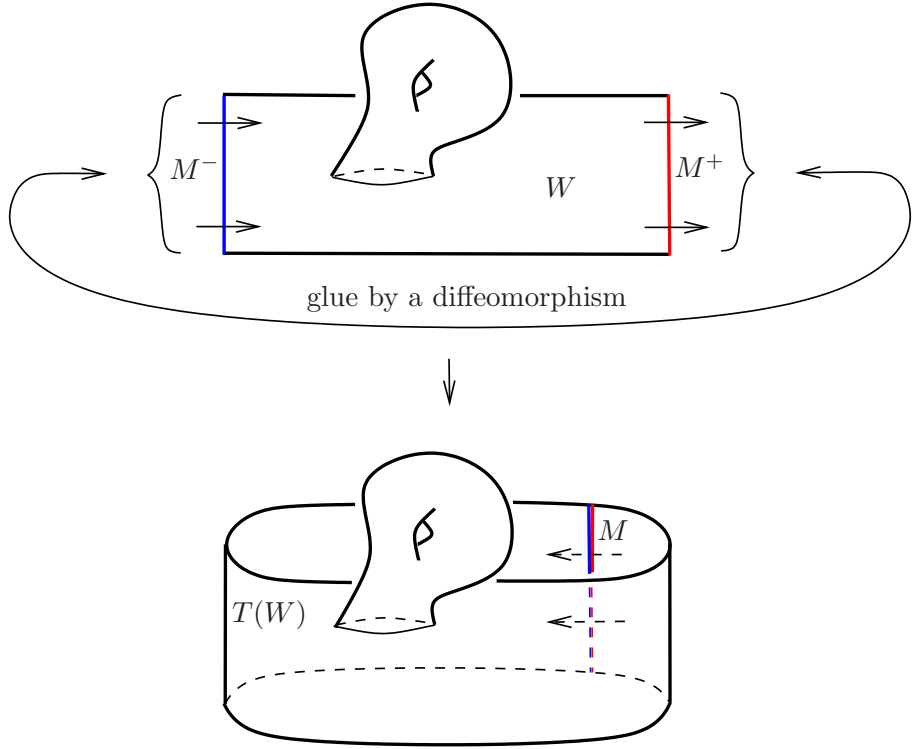}}}

\relabel{1}{glue by a diffeomorphism}

\relabel{2}{$W$}

\relabel{3}{$T(W)$}

\relabel{4}{$M^-$}

\relabel{5}{$M^+$}

\relabel{6}{$M$}

\endrelabelbox
        \caption{Mapping torus of an endobordism}
        \label{C}
\end{figure}

We will be mainly concerned with the following types of
endobordisms:

\begin{definition}  \label{def:cylcob}
       Let $M$ be a manifold-with-boundary.
       A {\bf cylinder} with {\bf base} $M$ is a trivial cobordism
       $W = I \times M$,
       the incoming boundary being $\partial_- I  \times M$ and
       the outgoing one being $\partial_+ I \times M$.
       A {\bf cylindrical cobordism} with {\bf base} $M$ is a cobordism $W$
       from a copy $M^-$ of $M$ to another copy $M^+$ such that  the
       union of connected components of $\partial W$  which intersect
       $M^- \cup M^+$---the {\bf cylindrical boundary}  $\partial_{cyl} W$---is endowed
       with a diffeomorphism (respecting the incoming and outgoing boundary regions)
       to the boundary $\partial( I \times M)= (\partial I \times M) \cup (I \times \partial M)$
       of a cylinder with base $M$.
     The segment $I$ is called the {\bf directing segment} of the cylindrical
       cobordism.
\end{definition}

Note that cylinders with base $M$ are special cases of cylindrical
cobordisms with base $M$, which are special cases of endobordisms of
$M$.

The composition of two cylinders/cylindrical cobordisms with the
same base $M$ is a cylinder/cylindrical cobordisms with base $M$.
 More generally, the composition of two endobordisms of $M$ is again
an endobordism of $M$.

To any \emph{cooriented} and \emph{properly embedded} hypersurface
$M$ of a (not necessarily oriented or even orientable)
manifold-with-boundary is associated canonically (up to
diffeomorphisms) an endobordism with base $M$.

\begin{definition}  \label{def:twocob}
    Let $W$ be a compact manifold-with-boundary and
    let $M \hookrightarrow W$ be a \emph{cooriented}  and \emph{properly embedded}
    compact
    hypersurface-with-boundary.  We view a collar neighborhood
     $[-1, 1] \times M \hookrightarrow W$ of $M$  as the cylinder
    $ Z_{[-1, 1]} : \{-1\} \times M \Longmapsto \{ + 1\} \times M$.  Denote by
    $Z_{[-1, 0]}$ and $Z_{[0, 1]}$ the analogous cylinders corresponding
    to the intervals $[-1, 0]$ and $[0, 1]$, which implies that
    $Z_{[-1, 1]} \simeq Z_{[0, 1]} \circ Z_{[-1, 0]}$.
    Let  $W_M$ be  the
    closure inside $W$ of the complement $W \setminus ([-1, 1] \times M)$.
    We see it as an endobordism  $W_M: \{1\} \times M \Longmapsto \{-1 \} \times
    M$, hence
   the composition $Z_{[-1, 0]} \circ W_M \circ Z_{[0, 1]}$
    is also an endobordism of $M$.
    We call this endobordism the {\bf splitting of $W$ along $M$}
    and denote it  by:
       \[ \Sigma_M(W): M^-    \Longmapsto M^+  \]
    (see Figure \ref{D}), where $M^\mp$ are two copies of $M$.
The natural map:
       \[ \sigma_M: \Sigma_M(W) \to W \]
     is called the {\bf splitting map} of $W$ along $M$ or of $M \hookrightarrow W$.
\end{definition}

\begin{figure}[ht]
  \relabelbox \small {
  \centerline{\epsfbox{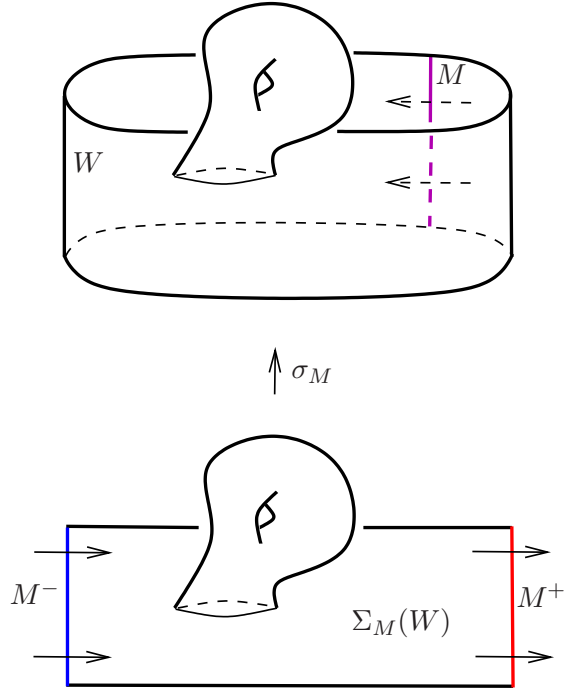}}}

\relabel{1}{$M$}

\relabel{3}{$W$}

\relabel{2}{$\Sigma_M(W)$}

\relabel{4}{$M^-$}

\relabel{5}{$M^+$}

\relabel{6}{$\sigma_M$}

\endrelabelbox
        \caption{Splitting of $W$ along a cooriented properly embedded hypersurface $M$ }
        \label{D}
\end{figure}

Intuitively, one modifies $W$ replacing each point of $M$ by the set
of two orientations of the normal line to $M$ at that point.

\begin{remark} \label{rem:gencut}
    \begin{enumerate}
       \item  \label{rem:orblow}
    The splitting map $\sigma_M$ is a diffeomorphism above $W \setminus M$,
    the preimage of $M$ by $\sigma_M$ being the disjoint union $M^+ \sqcup M^-$ of two
copies of $M$, distinguished canonically as the incoming and the
outgoing boundaries of the cobordism $\Sigma_M(W) : M^- \Longmapsto
M^+$. Both figures \ref{C} and \ref{D} may be seen as graphical representations of
the splitting map  $\sigma_M$. In the first case one starts from the source and in the second
case from the target, before constructing the map $\sigma_M$.

      \item  \label{rem:riemrel}
          The splitting map $\sigma_M$ allows one to
prove that the splitting of $W$ along $M$ is unique up to a unique
diffeomorphism above $W$ (that is, any two such splittings are
related by a unique diffeomorphism compatible with their splitting
maps). One may see $\Sigma_M(W)$ as a generalization of the surface
obtained by cutting a given surface along a properly embedded arc,
an operation fundamental in Riemann's approach of \cite{R 51} to the
topology of surfaces. Another way to model this splitting operation
is to remove a collar neighborhood of $M$. We preferred the previous
definition because of its canonical nature.

     \item  \label{rem:gensplit}
        One could also define a splitting map along
        non-coorientable hypersurfaces. In this case one would not obtain a
        cobordism, because above $M$ the map would restrict to a non-trivial covering
        of degree $2$. We did not define such splittings because we do not
        use them in this paper.
  \end{enumerate}
\end{remark}

We have the following immediate relation between the operations of
taking the mapping torus and of splitting:

\begin{proposition} \label{prop:recext}
     The operations of taking the mapping torus of an endobordism and
     of splitting along a cooriented properly embedded hypersurface
     are inverse to each other.
\end{proposition}

\bigskip
\section{Seifert hypersurfaces and open books}
\label{sect:Seifsect}

In this section we introduce a notion of \emph{Seifert hypersurface} and we
explain  in which sense it is equivalent to the notion of cylindrical cobordism introduced
  in the previous section. We conclude by treating the special case of Seifert
  hypersurfaces which are pages of open books.
  \medskip

Assume that $M$ is still a cooriented compact
hypersurface-with-boundary in $W$, but which is \emph{not} properly
embedded. Instead, we require $M$ \emph{to be contained in the
interior of $W$}. In order to write more concisely, we introduce a
special name for such hypersurfaces:

\medskip
\begin{definition}  \label{def:seifhyp}
   Let $W$ be a manifold-with-boundary.
   A compact hypersurface-with-boundary   $M \hookrightarrow W$ is  a
   {\bf Seifert hypersurface} if:
      \begin{itemize}
          \item the boundary of each connected component of $M$ is non-empty;
          \item $M \hookrightarrow \mbox{int}(W)$;
          \item $M$  is cooriented.
      \end{itemize}
\end{definition}

\begin{remark}
   Traditionally,  a \emph{Seifert surface} is defined as an oriented surface
   embedded in $\bS^3$, whose boundary is an oriented link $L$ which one wants
   to study. Seifert surfaces are often used algebraically through their
   associated \emph{Seifert forms}. To define the Seifert form,
   one needs to choose a positive side of the Seifert surface, to push some $1$-cycles off the
   surface towards that direction and to compute some linking numbers.
   An important ingredient in this construction is the
   \emph{coorientation} of
the Seifert surface, which is canonically  determined by the
orientation of $L$ and $\bS^3$.
   For this reason, we have decided to extend this aspect of Seifert surfaces
   in $\bS^3$ to a general definition, that also subsumes Lines' Definition
   \ref{def:simpleknot}.
 \end{remark}

There is a canonical way to associate to a Seifert hypersurface $M$
of $W$ a \emph{cooriented} and \emph{properly embedded}
hypersurface-with-boundary in a new manifold (see Definition
\ref{def:extnonprop}). But in order to achieve this, one has first
to ``pierce'' $W$ along $\partial M$. We will define this piercing
procedure using special trivialized tubular neighborhoods of
$\partial M \hookrightarrow W$:

\begin{definition}  \label{def:prodneighb}
   Let $W$ be a manifold-with-boundary and let $M \hookrightarrow W$
   be a Seifert hypersurface.
   A tubular neighborhood $\mathcal{U}_W(\partial M)$ of
   $\partial M \hookrightarrow W$ is called {\bf adapted to} $M$ if it is
   endowed with a product structure $\bD^2 \times \partial M$ such that
   $M$ intersects it along $[0,1] \times \partial M$ (where $[0, 1] \hookrightarrow \bD^2$ is the
   canonical embedding) and if the canonical
   orientation of $\partial \bD^2$ coincides with the given coorientation
   of $M$ in $W$. The composition of the first projection
   $\mathcal{U}_W(\partial M) \setminus \partial M
   \simeq (\bD^2 \setminus \{0\}) \times \partial M \to \bD^2 \setminus \{0\}$
   with the angular coordinate $\theta: \bD^2 \setminus \{0\}  \to \bS^1$
   is called an {\bf angular coordinate of $\partial M$ adapted to $M$}
   (see Figure \ref{E}).
\end{definition}

\begin{figure}[ht]
  \relabelbox \small {
  \centerline{\epsfbox{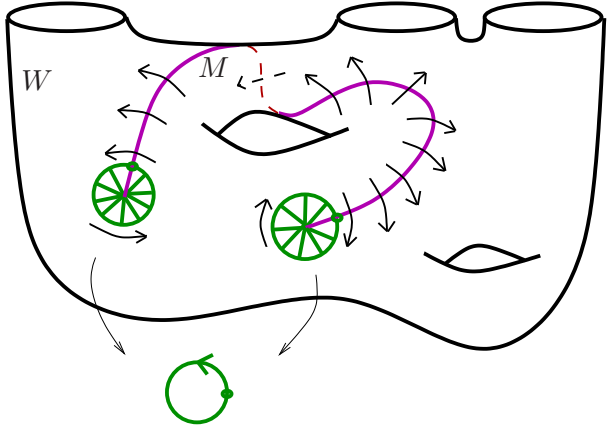}}}

\relabel{1}{$M$}

\relabel{2}{$W$}

\endrelabelbox
        \caption{Angular coordinate of $\partial M$ adapted to $M$}
        \label{E}
\end{figure}

An adapted tubular neighborhood of the boundary of a Seifert
hypersurface always exists and is unique up to isotopy. The
reason is that the normal bundle $\mathcal{N}_{\partial M \mid W}$ of
$\partial M$ in $W$ is canonically trivialized up to homotopy, by
taking as a first section a nowhere vanishing incoming vector field
on $M$ along $\partial M$ and as an independent section a positively
normal vector field of $M$ along $\partial M$ (recall the
fundamental hypothesis that $M$ is \emph{cooriented}).

We want to pierce or blow-up $W$ in an oriented way along $\partial
M$. We will define this operation using the following local model to
be used in each fiber of an adapted tubular neighborhood:

\begin{definition}   \label{def:polblow}
    The {\bf radial blow-up} of $\bD^2$ is the map $\pi_0: [0, 1] \times \bS^1 \to \bD^2$
    which expresses the cartesian coordinates on  $\bD^2$ in terms of polar ones:
       \[   (r, \theta) \mapsto (r \cos \theta, r \sin \theta). \]
\end{definition}

One may perform the radial blow-up operation fiberwise in an
adapted tubular neighborhood of a Seifert hypersurface:

\begin{definition}  \label{def:polarblow}
    Let $W$ be a manifold-with-boundary and let $M \hookrightarrow W$
   be a Seifert hypersurface. Let $\bD^2 \times \partial M  \hookrightarrow W$
   be a tubular neighborhood of $\partial M$ adapted to $M$. Let $\Pi_{\partial M}(W)$
   be the manifold obtained as the union
   of  $W \setminus \partial M$ and $[0, 1] \times \bS^1 \times \partial
   M$, where  $(\bD^2 \setminus \{0\}) \times \partial M$ in  $W \setminus \partial M$ is
   identified with $(0, 1] \times \bS^1 \times \partial M$ in
   $[0, 1] \times \bS^1 \times \partial M$ through
   the diffeomorphism $\pi_0 \times \mbox{id}_{\partial M}$.
   The {\bf radial blow-up of $W$ along $\partial M$} is the map:
    \[Ê\pi_{\partial M}:  \Pi_{\partial M}(W) \to W \] described as follows:  $\pi_{\partial M}$
    is just the the inclusion map
    on $W \setminus \partial M$ and is given by  $\pi_0 \times \mbox{id}_{\partial M}$
    on $[0, 1] \times \bS^1 \times \partial
    M$. We also say that  $\Pi_{\partial M}(W)$ is obtained by
    {\bf piercing $W$ along
   $\partial M$}. The {\bf strict transform} $M'$ of $M$ by $\pi_{\partial M}$ is
    the closure of $(\pi_{\partial M})^{-1}(\mathrm{int}(M))$ inside  $\Pi_{\partial M}(W)$.
\end{definition}

The operation of radial blow-up is also called \emph{oriented
blow-up} in the literature, but under that name it is in general
used in the semi-algebraic category. Intuitively, $W$ is modified by
replacing each point of $\partial M$ by the circle of oriented lines
passing through the origin of the normal plane to $\partial M$ at
that point. We have the following easy lemma:

\begin{lemma}  \label{lem:proptransf}
    The radial blow-up map $\pi_{\partial M}$ is proper and a diffeomorphism
    above  $W \setminus M$.
    The restriction $\pi_{\partial M}|_{M'} : M' \to M$ is a diffeomorphism.
\end{lemma}

In the sequel,
    we will identify $M$ and $M'$ using this diffeomorphism, which will allow us to
    speak about the embedding $M \hookrightarrow  \Pi_{\partial M}(W)$. This embedding
    is cooriented (by the lift of the coorientation of $M$ in $W$) and proper.

\begin{figure}[ht]
  \relabelbox \small {
  \centerline{\epsfbox{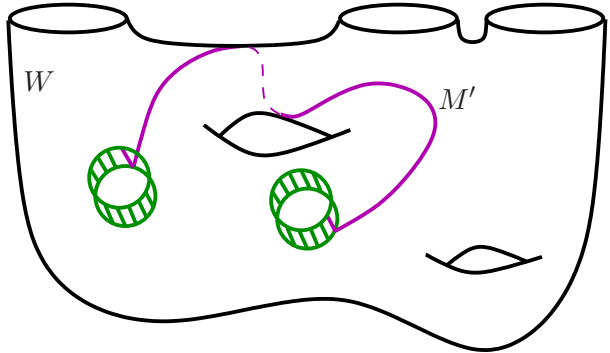}}}

\relabel{1}{$M'$}

\relabel{2}{$W$}

\endrelabelbox
        \caption{The radial blow-up of the surface $W$ of Figure \ref{E}
        along the end points of the arc $M$, and the strict transform $M'$ of $M$.
      The two green cylinders are attached transversely to $W$ (in the usual $3$-dimensional space),
      where we first remove two disks from $W$ (colour figure online).}
        \label{Eblow}
\end{figure}

\begin{remark}   \label{rem:abpol}
   \begin{enumerate}
      \item \label{rem:transprop}
          The radial blow-up  allows us to pass from a Seifert hypersurface
          to a properly embedded cooriented hypersurface in the pierced manifold.
      \item \label{rem:genpol}
           This remark is to be compared with Remark \ref{rem:gencut} (\ref{rem:gensplit}).
           One could define an analogous operation of radial blow-up along
          an arbitrary submanifold of codimension $2$, as one does not need
          to have a globally trivial fibered tubular neighborhood in order to do
          fiberwise radial blow-ups of the centers of the discs. Nevertheless, we
          introduced this more restricted definition, as the only one which is
          needed in the paper.
   \end{enumerate}
 \end{remark}

As $M$ is cooriented and properly embedded in  $\Pi_{\partial M}(W)$, one may consider
the splitting $\Sigma_{M}( \Pi_{\partial M}(W))$, as introduced in Definition
\ref{def:twocob}:

\begin{definition}  \label{def:extnonprop}
     Let $W$ be a manifold-with-boundary and let $M \hookrightarrow W$
   be a Seifert hypersurface.  The {\bf splitting of $W$ along $M$}, denoted
   $\Sigma_M(W)$, is defined as the splitting  $\Sigma_{M}( \Pi_{\partial M}(W))$
   of the properly embedded
   hypersurface $M \hookrightarrow  \Pi_{\partial M}(W)$.
 It is therefore  an endobordism of $M$. The composition
 $\pi_{\partial M} \circ \sigma_M :  \Sigma_M(W) \to W$ is called the {\bf splitting map}
 of $W$ along $M$   (Fig. \ref{Esplit}).
\end{definition}

\begin{figure}[ht]
  \relabelbox \small {
  \centerline{\epsfbox{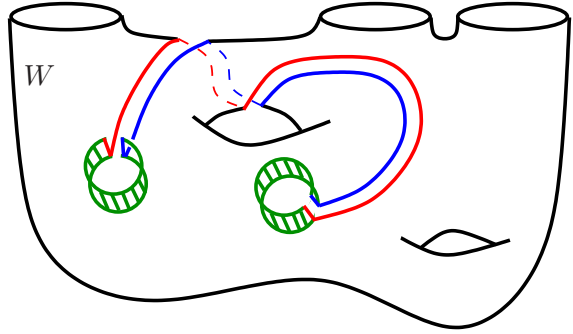}}}

\relabel{2}{$W$}

\endrelabelbox
        \caption{The splitting of $W$
along $M$ for the pair $(W,M)$ of Figure \ref{E}. Recall that the
intermediate radial blow-up is drawn in Figure \ref{Eblow}.}
        \label{Esplit}
\end{figure}

\begin{remark}   \label{rem:cylcob}
    This remark is a continuation of Remark \ref{rem:gencut} (\ref{rem:riemrel}).
    Riemann explained that one has to cut a surface along a curve which goes
    from the boundary to the boundary. This is the operation we modeled in arbitrary
    dimensions by Definition \ref{def:twocob}. He added that if the surface has
    no boundary, then one has first to pierce it, creating like this an infinitely small
    boundary, and then one may cut it along a curve going from this boundary to
    itself. This is the operation we modeled in arbitrary dimensions in Definition
    \ref{def:extnonprop}. We first ``pierced'' $W$ along the boundary of $M$
    (Definition \ref{def:polarblow}), and then we were able to
    apply Definition  \ref{def:twocob}.
\end{remark}

One has the following immediate observation, consequence of the fact
that one gets a segment by splitting a circle at a point. This
observation is nevertheless very important for the sequel. Recall
that both notions of cylindrical cobordisms and of their cylindrical
boundaries were introduced in Definition \ref{def:cylcob}:

\begin{lemma}  \label{lem:obscyl}
      If $M \hookrightarrow W$ is a Seifert hypersurface, then the
      splitting $\Sigma_M(W)$ is a cylindrical cobordism whose cylindrical boundary
      is given by:
          \[ \sigma_M^{-1} (M \cup \pi_{\partial M}^{-1}(\partial   M)).\]
\end{lemma}

Assume conversely that $W : M^- \Longmapsto M^+$ is a cylindrical
cobordism with base $M$, its cylindrical boundary being identified
with $\partial( I \times M)$. Fix an orientation-preserving
identification of $\bS^1$ with the circle obtained from $I$ by
 gluing $\partial_- I$ and $\partial_+ I$. One identifies therefore
to $\bS^1 \times \partial M$ the image of the cylindrical boundary inside the mapping torus
$T(W)$. This allows us to define:

\begin{definition} \label{def:circol}
  Let $W : M^- \Longmapsto M^+$ be a cylindrical cobordism with base $M$.
  Its {\bf circle-collapsed mapping torus} $T_c(W)$ is obtained from the
  mapping torus $T(W)$ by collapsing the circle $\bS^1 \times \{m \}$ to
  $\{0\} \times  \{m \}$, for all $m \in \partial M$.
  The {\bf Seifert hypersurface associated to the cylindrical cobordism $W$} is
  the natural embedding $M  \hookrightarrow T_c(W)$.
\end{definition}

We have the following analog of Proposition \ref{prop:recext}:

\begin{proposition}  \label{prop:invop}
    The operations of taking the circle-collapsed mapping torus
    of a cylindrical cobordism and of splitting along a Seifert  hypersurface
     are inverse to each other.
\end{proposition}

This shows that, in differential-topological constructions, one may
use interchangeably  either cylindrical cobordisms or Seifert
hypersurfaces.

One may describe the construction of the circle-collapsed mapping
torus of a cylindrical cobordism in a slightly different way, by
filling the boundary of the mapping torus with a product manifold,
instead of collapsing the circles contained in it:

\begin{lemma}  \label{lem:secdescr}
      Let $W : M^- \Longmapsto M^+$ be a cylindrical cobordism with base $M$.
     The manifold obtained by gluing the mapping torus $T(W)$ to the product
     $\bD^2 \times \partial M$ through the canonical identification of their boundaries
     is diffeomorphic to the circle-collapsed mapping torus $T_c(W)$ through a
     diffeomorphism which is the identity on the complement of an arbitrary neighborhood
     of $\bD^2 \times \partial M$ and which sends $0 \times \partial M$ onto
     $\partial M$.
\end{lemma}

We will use this second description in the proof of Proposition
\ref{prop:samedef}.

\medskip
We apply now the previous considerations to the special situation where
$M \hookrightarrow W$ is a page of an \emph{open book}. Let us recall
first this notion:

\begin{definition} \label{def:openbook}
   An {\bf open book} in a closed manifold $W$ is a pair  $(K,  \theta)$ consisting of:
\begin{enumerate}
     \item a codimension $2$ submanifold $K \subset W$, called the {\bf binding},
        with a trivialized normal bundle;
     \item  a fibration $\theta: W \setminus K \to \bS^1$ which, in a tubular neighborhood
         $\bD^2 \times K $ of $K$ is the normal angular coordinate (that is, the composition
        of the first projection $\bD^2 \times K \to \bD^2 $ with the angular coordinate
        $\bD^2 \setminus \{0\}  \to \bS^1$).
 \end{enumerate}
\end{definition}

It follows that for each $\theta_0 \in \bS^1$, the closure in $W$ of
$\theta^{-1}(\theta_0)$---called a {\bf page} of the open book---is a
Seifert hypersurface whose boundary is the binding $K$. Its
coorientation is defined by turning the pages in the positive sense
of $\bS^1$. If $v$ is a vector field which is transverse to the
pages, meridional near $K$ and such that its vectors project to
positive vectors on $\bS^1$, then the first return map of $v$ on an
arbitrary page is called the \textbf{geometric monodromy} of the
open book. As in the $3$-dimensional case, such a geometric
monodromy is well-defined up to isotopies relative to the boundary
and conjugations by diffeomorphisms which are the identity on the
boundary.  No conjugation appears if the initial page is fixed.

One may describe the previous monodromical considerations in a slightly
different way, using the splitting of the ambient manifold along a page
(see Definition \ref{def:extnonprop}).
Let $M \hookrightarrow W$ be an arbitrary page of the open book.
The splitting of $W$ along $M$ is a cylindrical cobordism $\Sigma_M(W): M \Longmapsto M$.
Consider the same vector field as before. Its flow realizes a diffeomorphism
from the incoming boundary (a copy of $M$) to
the outgoing boundary (another copy of $M$).
Therefore it gives a diffeomorphism of $M$, which is moreover
fixed on the boundary of $M$. It is the geometric monodromy diffeomorphism!

This geometric monodromy is isotopic to the identity if and only if
$\Sigma_{M}(W)$ is isomorphic to the cylinder $I \times M$ by an
isomorphism which is the identity on the boundary and respects the
fibrations over the interval $I$. Note that $\Sigma_{M}(W)$ is
always isomorphic to that cylinder, if we do not impose constraints
on the boundary.

Conversely, for any self-diffeomorphism $\phi$ of a compact
manifold-with-boundary $M$ which is the identity on $\partial M$,
one can construct as follows a closed manifold equipped with an
open book with page $M$ and monodromy $\phi$:
\begin{itemize}
      \item  take the cylinder $[0, 1] \times M$;
      \item consider it as a cylindrical cobordism $[0, 1] \times M : M_0 \Longmapsto M_1$
        where $M_0, M_1$ are two copies of $M$, that $M_0$ is identified to
        $\{0 \} \times M$ using the identity of $M$ and $M_1$ is identified to
        $\{1 \} \times M$ using $\phi : M \simeq M_1$;
     \item take the circle-collapsed mapping torus associated to this cylindrical cobordism
         (see Definition \ref{def:circol}). The fibers of the first projection
         $[0, 1] \times M \to [0, 1]$ induce the pages of an open book structure on it.
  \end{itemize}

\begin{remark}\label{rem:eqfib}
    \begin{enumerate}
       \item \label{rem:aob}
            The pair $(M, \phi)$ is sometimes called an {\bf abstract open book}.

       \item  \label{rem:classic}
          The mapping torus of the previous cylindrical cobordism (according
           to Definition \ref{def:maptor}) coincides with the classical mapping torus $\cM(M, \phi)$
           of the diffeomorphism $\phi$. This is the reason why we chose the
           name ``mapping torus'' for the object introduced in Definition \ref{def:maptor}.

       \item  A codimension $2$ closed submanifold $K \hookrightarrow V$
            of a closed manifold is called a \textbf{fibered knot} if it is the
            binding of some open book $(K, \theta)$. In this case, the map
            $\theta$ is not part of the structure.

         \item \label{rem:histob}
               One may consult \cite{W 98} for a survey of the use of open books till 1998.
               Since then, Giroux's paper \cite{Gi 02} started a new direction of applications
               of open books, into contact topology.
               The expression ``open book" appeared for the first time in 1973 in
            the work of Winkelnkemper \cite{W 73}. Before, equivalent
            notions of ``fibered knots" and ``spinnable structures'' were
            introduced in 1972 by Durfee and Lawson \cite{DL 72} and Tamura
            \cite{T 72} respectively.  All those papers were partly inspired by Milnor's
            discovery in \cite{M 68} of such structures --- without using any name for them ---
            associated to any germ $f : (\C^n, 0) \to (\C, 0)$ of polynomial with
            an isolated singularity at $0$.
            In \cite{CNP 06} was introduced the name ``Milnor open book'' for the open books
            associated more generally to holomorphic functions on germs of complex spaces,
            when both have isolated singularities.
      \end{enumerate}
\end{remark}

\bigskip
\section{Abstract and embedded summing}
\label{sec:absum}

In this section we define a notion of \emph{sum} of
manifolds-with-boundary of the same dimension (see Definition
\ref{def:absum}), which generalizes the usual notion of
\emph{plumbing} recalled in Definition \ref{def:plumb}. The sum is
done along identified \emph{patches} and extends to a commutative
and associative operation on \emph{patched manifolds} with
identified patches. Then we define an embedded version of this sum
(see Definition \ref{def:embsumpatch}). Unlike the abstract sum, this
operation is in general non-commutative, but it is still associative
(see Proposition \ref{prop:propembsum}). It generalizes both
Stallings' and Lines' plumbing operations recalled in Section
\ref{sec:classop}.
\medskip

In the sequel, we will consider embedded
submanifolds-with-boundary in other mani\-folds-with-boundary of the
same dimension, where part of the boundary of the submanifold
belongs to the interior, and part to the boundary of the ambient
manifold. The next two definitions will allow us to speak shortly
about such embeddings:

\begin{definition}  \label{def:attastruc}
    Let $P$ be a compact manifold-with-boundary.
    An {\bf attaching region} $A \hookrightarrow \partial P$ is
    a compact  manifold-with-boundary of the same dimension as
    $\partial P$.
    The closure $B := \overline{\partial P \setminus A}$ of the
    complement of the attaching region is the {\bf non-attaching region}.
    We say that $(P, A)$ is an {\bf attaching structure} on $P$.
    The {\bf complementary attaching structure} of $(P,A)$  is $(P, B)$.
\end{definition}

\begin{definition}   \label{def:patch}
   Let $M$ be an $n$-dimensional compact  manifold-with-boundary.
   A {\bf patch} of $M$ is the datum of an attaching structure $(P, A)$ on
   another  $n$-dimensional compact manifold-with-boundary
   and of an embedding $P \hookrightarrow M$ such that
   $P \cap \partial M = B$, where $B$ is the non-attaching region of
   $(P,A)$ (see Figure \ref{F}). That is, the attaching region $A$ is the closure
   of $\partial P \cap \mbox{int}(M)$.
   A manifold endowed with a patch is a {\bf patched manifold}.
   We denote it either as a pair $(M, P)$ or as an embedding
   $P \hookrightarrow M$.
\end{definition}

\begin{figure}[ht]
  \relabelbox \small {
  \centerline{\epsfbox{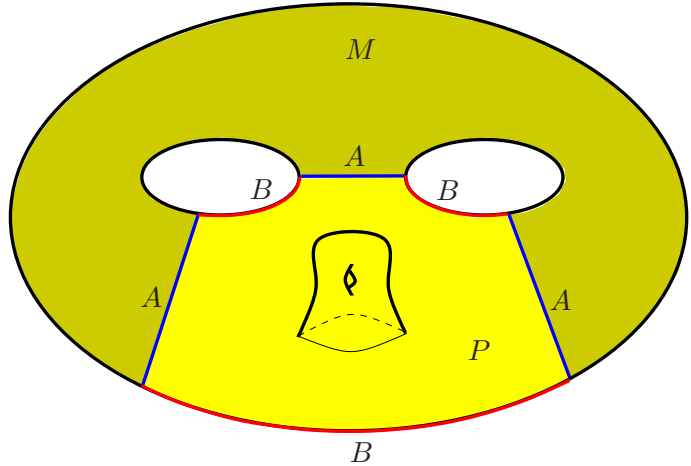}}}

\relabel{1}{$A$} \relabel{2}{$A$} \relabel{3}{$A$}

\relabel{4}{$B$} \relabel{5}{$B$} \relabel{6}{$B$}

\relabel{7}{$P$} \relabel{8}{$M$}

\endrelabelbox
        \caption{A patched manifold $(M, P)$ with patch $(P,A)$}
        \label{F}
\end{figure}

\begin{remark}
    \begin{enumerate}
        \item
      The condition $P \  \cap \  \partial M = B$   is equivalent to the condition that
      the attaching region $A$ is the closure  of $\partial P \cap \mbox{int}(M)$.
      Therefore, the attaching region is determined by the embedding
      $P \hookrightarrow M$. We chose the name ``\emph{attaching region}''
      thinking to the fact that $P$ is attached to $\overline{M \setminus P}$
      along it.

          \item  As represented in Figure \ref{F}, a patch $(P,A)$ is best
          thought as a manifold with corners. When we speak about $P$ as a
          manifold-with-boundary, we again use implicitly the fact, recalled at the beginning
          of Section \ref{sec:notconv}, that the corners may be smoothed.
      \end{enumerate}
\end{remark}

Now we are ready to give the main definition of this section, that
of an operation of summing of two patched manifolds with identified
patches:

\begin{definition}   \label{def:absum}
     Let $M_1$ and $M_2$ be two compact manifolds-with-boundary of the
     same dimension. Assume that a manifold $P$ is a patch
     of both $M_1$ and $M_2$, with the corresponding attaching regions $A_1$ and
     $A_2$, such that $A_1 \cap A_2= \emptyset$. Then we say
     that the two patched manifolds $(M_1, P)$ and
     $(M_2, P)$ are {\bf summable}.
     The {\bf (abstract) sum of $M_1$ and $M_2$ along $P$}, denoted by:
        \[   M_1 \biguplus^{P} M_2, \]
     is the compact manifold-with-boundary obtained from the disjoint union
     $M_1 \bigsqcup M_2$ by gluing the points of both copies of $P$
     through the identity map. Its {\bf associated patch} is the canonical embedding
     $\displaystyle{P \hookrightarrow M_1 \biguplus^{P} M_2}$, obtained by
     identifying the two given patches with attaching region $A_1 \cup A_2$
     (see Figure \ref{K}).
\end{definition}

\begin{figure}[ht]
  \relabelbox \small {
  \centerline{\epsfbox{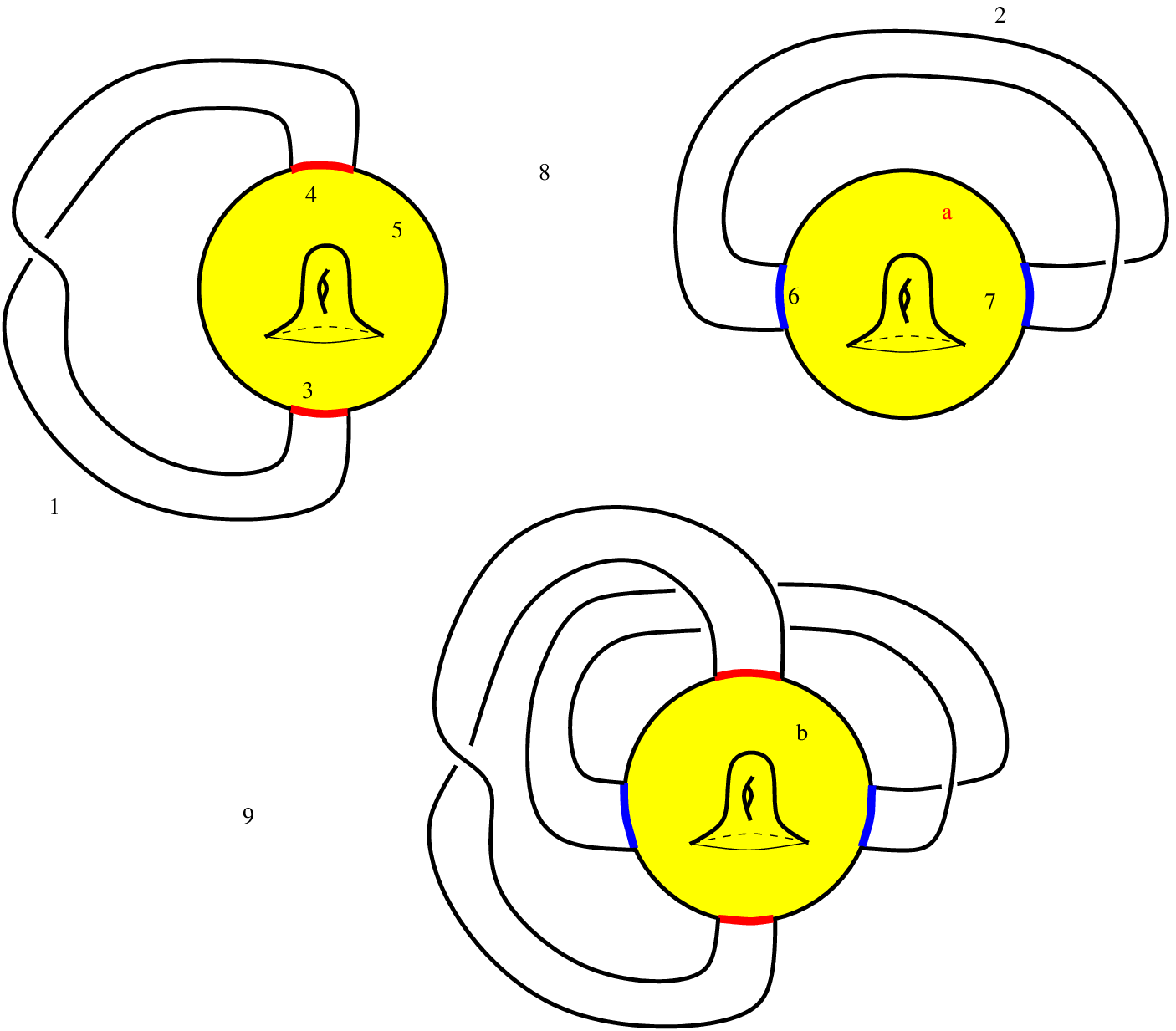}}}

\relabel{1}{$M_1$} \relabel{2}{$M_2$} \relabel{3}{$A_1$}

\relabel{4}{$A_1$}

\relabel{5}{$P$} \relabel{a}{$P$} \relabel{b}{$P$}

\relabel{6}{$A_2$}

\relabel{7}{$A_2$}

\relabel{8}{$\displaystyle{\biguplus^{P}}$}

\relabel{9}{$=$}

\endrelabelbox
        \caption{The abstract sum $\displaystyle{M_1 \biguplus^{P} M_2}$
           of $M_1$ and $M_2$ along $P$
       }
        \label{K}
\end{figure}

Note that Definition~\ref{def:absum} respects our convention
explained in Remark \ref{rem:simplnot}. It may be
immediately extended to the case where the patches are distinct, and
are identified by a given diffeomorphism, such that after the
identification the attaching regions are disjoint.

\begin{remark}  \label{rem:varem}
    \begin{enumerate}
       \item  \label{rem:disj}
             The attaching region of $\displaystyle{P \hookrightarrow M_1 \biguplus^{P} M_2}$
             is the union of the attaching regions of
             $P \hookrightarrow M_1$ and $P \hookrightarrow M_2$.

        \item \label{rem:welldef}
           One may also present the construction of $\displaystyle{M_1 \biguplus^{P} M_2}$ in
           the following way (see Figure \ref{L}):
           glue $\overline{M_1 \setminus P}$ to $M_2$ by the canonical identification of
           $A_1 \hookrightarrow \partial ( \overline{M_1 \setminus P})$
           and $A_1 \hookrightarrow \partial M_2$. One has this last inclusion
           because the hypothesis $A_1 \cap A_2 = \emptyset$
           implies that $A_1 \subset B_2 \subset \partial M_2$, where $B_2$ denotes
           the non-attaching region of $(P, A_2)$. This second description shows
           that, indeed, the sum $\displaystyle{M_1 \biguplus^{P} M_2}$ has a structure of
           manifold-with-boundary. One has of course a symmetric description
           obtained by permuting the indices $1$ and $2$.

\begin{figure}[ht]
  \relabelbox \small {
  \centerline{\epsfbox{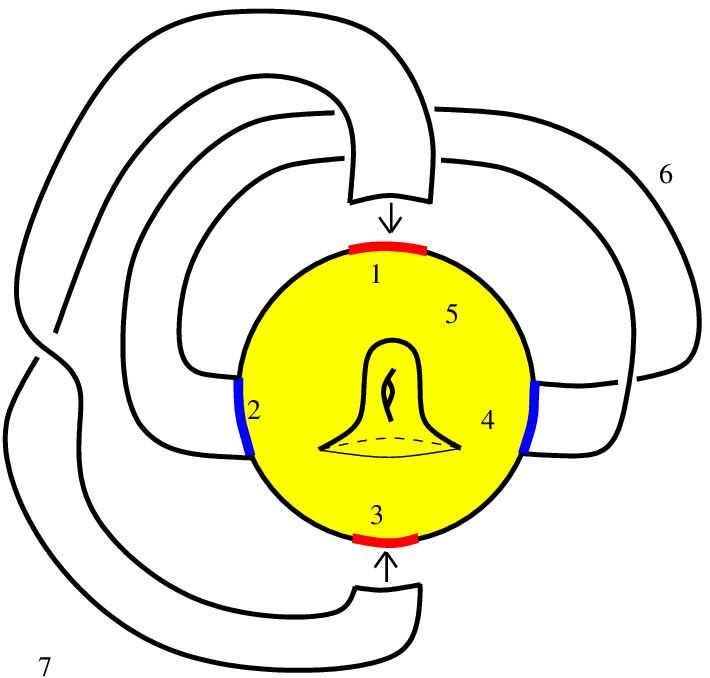}}}

\relabel{1}{$A_1$} \relabel{2}{$A_2$} \relabel{3}{$A_1$}

\relabel{4}{$A_2$}

\relabel{5}{$P$}

\relabel{6}{$M_2$}

\relabel{7}{$\overline{M_1 \setminus P}$}

\endrelabelbox
        \caption{An alternative description of the
             abstract sum $\displaystyle{M_1 \biguplus^{P} M_2}$
       }
        \label{L}
\end{figure}

         \item \label{rem:isosum}
               If  $\displaystyle{M_1 \biguplus^{P} M_2}$ is viewed as described in
               the previous remark, one can see that a diffeomorphic
               manifold is obtained by allowing isotopies of $A_1$ inside the non-attaching
               region $B_2 := \overline{\partial P \setminus A_2}$. In other words, it is
               sufficient to require only that \emph{the interiors of $A_1$ and $A_2$
               are disjoint}. Note that, if $A_1 \cap A_2 \neq \emptyset$, then strictly speaking,
               $P$ is not a patch inside $\displaystyle{M_1 \biguplus^{P} M_2}$.
               Nevertheless, in this case one still gets a canonical realization of $P$ as a patch,
               \emph{up to isotopy},
               in $\displaystyle{M_1 \biguplus^{P} M_2}$, by isotoping $A_1$ inside
               itself so that
               the hypothesis $A_1 \cap A_2 = \emptyset$ is
               achieved.
              As explained in Section \ref{sec:notconv}, the operations of gluing done here
               are defined up to smoothing of the corners.

          \item  \label{rem:genplumb}
              When the two patches used in the summing are the complementary
              patches $(\bD^n \times \bD^n, \bS^{n-1} \times \bD^n)$ and
              $(\bD^n \times \bD^n,  \bD^n \times \bS^{n-1})$, one gets
              the classical notion of \emph{plumbing} recalled in Definition
              \ref{def:plumb}. This is an example of a situation discussed in
              Remark \ref{rem:varem} (\ref{rem:isosum}), in which only the interiors of the
              attaching regions are disjoint.

        \end{enumerate}
\end{remark}

Remark \ref{rem:varem} (\ref{rem:isosum})  shows that one may define
the {\bf abstract sum}:
             \[  \biguplus^P_{i \in I} M_i  \]
            whenever $P$ appears  as a patch of all the manifolds in a finite collection
            $(M_i)_{i \in I}$ of manifolds-with-boundary with \emph{pairwise
            disjoint interiors of attaching regions}.

            This sum is \emph{commutative and associative}
            (up to unique isomorphisms), which motivates the absence of brackets in the notation.
            It is again endowed with a canonical patch $\displaystyle{P \hookrightarrow
             \biguplus^P_{i \in I} M_i}$ whenever the attaching regions themselves
             are pairwise disjoint.
             As explained in Remark~\ref{rem:varem} (\ref{rem:isosum}),
             if only the interiors of the initial patches are assumed to be
             disjoint, then there is still such a patch, but only well-defined up to
             isotopy.

 \medskip
  We pass now to the definition of the \emph{embedded sum}. Let us explain first
    which are the objects which may be summed in this way.

\begin{definition}  \label{def:embpatch}
       Let $W$ be a compact manifold-with-boundary and
        $P \hookrightarrow M$ be a patched manifold.
        Assume that $M \hookrightarrow \mbox{int}(W)$ is an embedding
        of $M$ as a hypersurface of $\mbox{int}(W)$.
        We say that the triple $(W, M, P)$, also denoted
       $P \hookrightarrow M \hookrightarrow W$, is
       a {\bf patch-cooriented triple}  if:
           \begin{itemize}
               \item[$\bullet$]  $P$ is coorientable in $W$;
                \item[$\bullet$]  a coorientation of $P$ in $W$ is chosen.
           \end{itemize}
\end{definition}

In the previous definition, $M$ is not necessarily a Seifert
hypersurface of $W$ (see Definition \ref{def:seifhyp}). Indeed, we
only assume that a coorientation was chosen \emph{along} $P$. It is even
possible that $M$ is not coorientable inside $W$. To illustrate
this, we depict in Figure~\ref{G} a cooriented quadrilateral patch $P$ of a M\"obius
band $M  \hookrightarrow W :=\bS^3$.

\begin{figure}[ht]
  \relabelbox \small {
  \centerline{\epsfbox{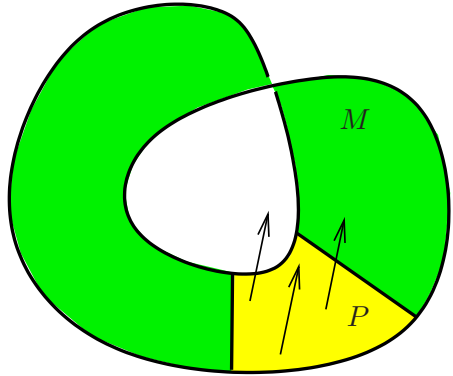}}}

\relabel{1}{$M$}

\relabel{2}{$P$}

\endrelabelbox
        \caption{Cooriented quadrilateral patch $P$ in a M\"obius band $M$}
        \label{G}
\end{figure}

Recall that the notion of \emph{positive side} for a cooriented
hypersurface was explained in Definition \ref{def:sides}. Let
$I^+$ and $I^-$ denote oriented compact intervals.

\begin{definition} \label{def:thickpatch}
    Let $(W, M, P)$ be a patch-cooriented triple. A {\bf positive thick patch}
    of $(W,M,P)$  is a choice of positive side $I^+ \times P \hookrightarrow W$
    of $P \hookrightarrow W$ intersecting $M$ only along $P$.
    If for example $I^+ =[0,1]$, then this means that $\{0\}\times P$
    maps to $P$ in $M$, and the positive tangents to $I^+$ point in
    the direction of co-orientation.
    Analogously, a {\bf negative thick patch}
    of $(W,M,P)$  is a choice of negative side $I^- \times P \hookrightarrow W$
    of $P \hookrightarrow W$, also intersecting $M$ only along $P$.
\end{definition}

We may now describe a generalization of Stallings' (embedded)
plumbing operation recalled in Section \ref{sec:classop} (see the
quotation containing equality (\ref {eq:plumbrel})) and of Lines'
higher dimensional plumbing operation (see Definition
\ref{def:msumhigh}):

\begin{definition}  \label{def:embsumpatch}
         Let $(W_1, M_1, P)$ and $(W_2, M_2, P)$ be
         two patch-cooriented triples with
         identified patches, such that $(M_1, P)$ and $(M_2, P)$ are two
         summable patched manifolds (recall Definition \ref{def:absum}).
        Then we say  that the two triples are {\bf summable} and
         their {\bf (embedded) sum}, denoted by:
           \[   (W_1,M_1) \biguplus^{P} \  (W_2, M_2),   \]
         is the compact manifold-with-boundary obtained by the following process
         (see Figure \ref{H}):
         \begin{itemize}
               \item choose a positive thick patch $I^+ \times P
                  \hookrightarrow W_1$ of $(W_1, M_1, P)$ and a negative thick patch
                  $I^- \times P \hookrightarrow W_2$ of $(W_2, M_2, P)$;

               \item consider the complements of their interiors
                   $W_1' := W_1 \setminus \mbox{int}(I^+ \times P)$
                   and $W_2' := W_2 \setminus \mbox{int}(I^- \times P)$;

                \item glue $W_1'$ to $W_2'$ by identifying
                  $\partial( I^+ \times P) \hookrightarrow W_1'$ to
                  $\partial( I^- \times P) \hookrightarrow W_2'$ through
                  the restriction of the map $\sigma \times \mbox{id}_P :
                  I^+ \times P \to I^- \times P$. Here
                  $\sigma : I^+ \to I^-$ denotes any diffeomorphism which reverses
                  the orientations (that is, such that $\sigma(\partial_{\pm} I^+) =
                    \partial_{\mp} I^-$).
         \end{itemize}
It follows that:
        $$\big(  (W_1,M_1) \biguplus^{P} \  (W_2,
M_2), M_1 \biguplus^{P} M_2, P\big)$$
is a patch-cooriented triple through the canonical embeddings:
        \[ P \hookrightarrow   M_1 \biguplus^{P} M_2 \hookrightarrow
               (W_1,M_1) \biguplus^{P} \:  (W_2, M_2)  \]
       and the gluing of the coorientations of $P$ in $W_1$ and in
       $W_2$.
\end{definition}

\begin{figure}[ht]
  \relabelbox \small {
  \centerline{\epsfbox{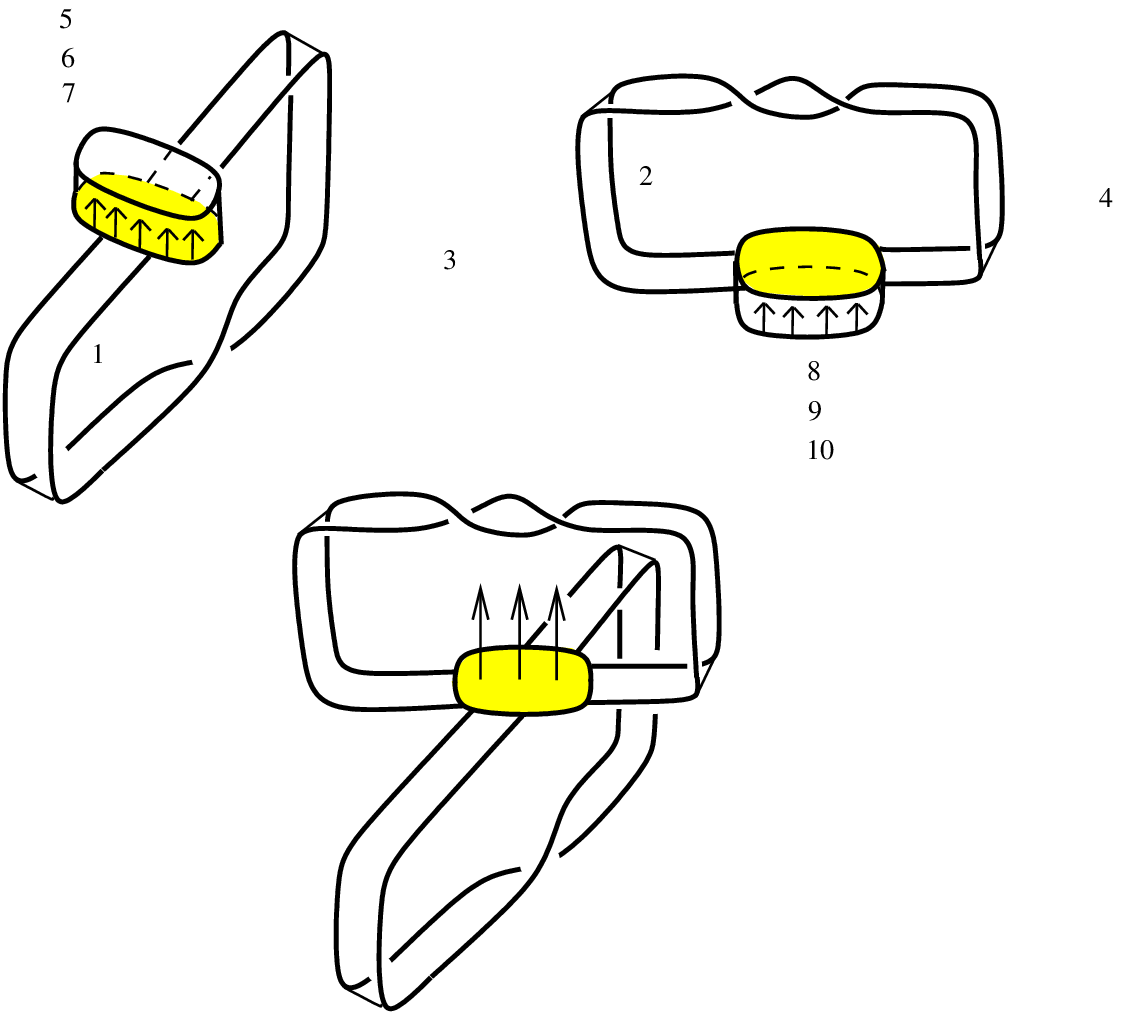}}}

\relabel{1}{$M_1$}

\relabel{2}{$M_2$}

\relabel{3}{$ \displaystyle{\biguplus^{P}}$}

\relabel{4}{$=$}

\relabel{5}{positive}

\relabel{6}{thick}

\relabel{7}{patch}

\relabel{8}{negative}

\relabel{9}{thick}

\relabel{10}{patch}

\endrelabelbox
        \caption{Embedded sum $\displaystyle{(W_1,M_1) \biguplus^{P} \:  (W_2, M_2)}$
            of two patch-cooriented triples}
        \label{H}
\end{figure}

\begin{remark}
    \begin{enumerate}
                   \item  The manifold $ \displaystyle{(W_1,M_1) \biguplus^{P} \:  (W_2, M_2)}$
          has non-empty boundary  if and only if either $W_1$ or $W_2$
          has a non-empty boundary.

          \item The abstract sum $\displaystyle{M_1 \biguplus^{P} M_2}$
            is obtained inside $\displaystyle{ (W_1,M_1) \biguplus^{P} \:  (W_2, M_2)}$
            as the union of the images of $M_1 \hookrightarrow W_1'$
            and of $M_2 \hookrightarrow W_2'$.

             \item We choose to take a positive thick patch for the first hypersurface
             and a negative one for the second hypersurface in order to respect
             Stallings' convention (see the citation containing formula
             (\ref{eq:plumbrel})). If we choose the other way
             around, we get an alternative definition of the embedded sum of the triples
             $(W_1, M_1, P),  \:  (W_2, M_2, P)$, which is
             diffeomorphic to $\displaystyle{ (W_2, M_2) \biguplus^{P} \:  (W_1, M_1)}$ by
             a diffeomorphism which fixes $\displaystyle{M_1 \biguplus^{P} M_2}$
             and the coorientation of $P$. The operation of embedded sum being
             in general non-commutative (see Proposition \ref{prop:propembsum}),
             this alternative definition is indeed different from Definition \ref{def:embsumpatch}.
    \end{enumerate}
\end{remark}

\begin{proposition}  \label{prop:propembsum}
     The patch being fixed, the operation of embedded sum of patch-cooriented triples
     is associative, but non-commutative in general.
\end{proposition}

\begin{proof}
    Let us prove first the \emph{associativity} of the operation.  Consider three
    summable patch-cooriented triples $(W_1, M_1, P), (W_2, M_2, P), (W_3, M_3, P)$,
    that is, assume that the attaching regions $A_1, A_2, A_3$ are pairwise disjoint.
    We want to prove that the two patch-cooriented triples:
       \[ \begin{array}{c}
              \displaystyle{(  \:   \Big( (W_1, M_1)  \biguplus^{P} \:  (W_2, M_2) \Big) \biguplus^{P}
                   \:  (W_3, M_3) ,
                 M_1  \biguplus^{P} \: M_2 \biguplus^{P} \: M_3, P  \:  ) },   \\
              \displaystyle{(\:    (W_1, M_1)  \biguplus^{P} \: \Big( (W_2, M_2)
                  \biguplus^{P} \:  (W_3, M_3) \Big) ,
                 M_1  \biguplus^{P} \: M_2 \biguplus^{P} \: M_3, P  \:  )  }
          \end{array}  \]
    are isomorphic.
    But this is an immediate consequence of the fact (see Definition \ref{def:embsumpatch})
    that both may be obtained from the disjoint union $W_1 \sqcup W_2 \sqcup W_3$ by removing:
        \begin{itemize}
            \item the interior of a positive thick patch of $(W_1, M_1, P)$;
            \item the interiors of a positive and of a negative thick patch of $(W_2, M_2, P)$,
                which intersect only along $P$;
            \item the interior of a negative thick patch of $(W_3, M_3, P)$;
        \end{itemize}
    and executing then the same gluings.

    \medskip

    Let us show now that the operation is \emph{non-commutative} in general.
       Consider the particular case where the triples to be summed are bands in
       $3$-spheres, as  in Figure \ref{H}, that is, $M_1$ and $M_2$ are
       either annuli or M\"obius bands. Moreover, assume that the patches are
       disks disposed as in that figure, that is, such that one may choose core circles
       $K_1, K_2$ of the two bands such that they intersect transversally once inside $P$.

       Denote by $J_1$ the arc of $K_1$ intercepted by $P$. Isotope $K_1$
       inside both  $\displaystyle{ (\bS^3, M_1) \biguplus^{P} \:  (\bS^3, M_2)}$ and
       $\displaystyle{ (\bS^3, M_2) \biguplus^{P} \:  (\bS^3, M_1)}$ by pushing the arc
       $J_1$ a little outside $P$ \emph{towards the positive side} of $P$, and keeping
       its complement in $K_1$ fixed. Denote by $K_1^+$ the new circle, contained
       either in $\displaystyle{ (\bS^3, M_1) \biguplus^{P} \:  (\bS^3, M_2)}$ or in
       $\displaystyle{ (\bS^3, M_2) \biguplus^{P} \:  (\bS^3, M_1)}$.  Look then at the
       linking number (modulo $2$) $\mbox{lk}(K_1^+, K_2)$. It is equal to
       $1$ in the first case and to $0$ in the second case.

          This shows that there is no isomorphism from
           $\displaystyle{ (\bS^3, M_1) \biguplus^{P} \:  (\bS^3, M_2)}$  to   \linebreak
           $\displaystyle{ (\bS^3, M_2) \biguplus^{P} \:  (\bS^3, M_1)}$ which is
           fixed on $\displaystyle{ M_1 \biguplus^{P} \:  M_2}$ and respects the coorientation
           of $P$. This is enough in order to deduce that the operation of embedded summing
           is in general non-commutative.
\end{proof}

In the next section we will consider carefully the special situation in which the hypersurfaces
$M_i \hookrightarrow W_i$ are globally cooriented:

   \begin{definition}  \label{def:patchseif}
         Let $(W_1, M_1, P)$ and $(W_2, M_2, P)$ be
         two patch-cooriented triples with identified patches. They are called
         {\bf summable patched Seifert hypersurfaces} if both $M_1 \hookrightarrow W_1$
         and $M_2 \hookrightarrow W_2$ are Seifert hypersurfaces whose coorientations
         extend those of the patches.
   \end{definition}

\bigskip
\section{The sum of stiffened cylindrical cobordisms}
\label{sec:embsum}

In Section~\ref{sec:absum} we defined an operation of embedded sum
for (summable) \emph{patch-cooriented} triples without assuming that
the hypersurfaces endowed with the (identified) patches are
themselves cooriented or even coorientable. In this section we will
assume this supplementary condition and we give an alternative
definition of the (embedded) sum based on the equivalence of Seifert
hypersurfaces and cylindrical cobordisms stated in Proposition
\ref{prop:invop}. In the next section we will show that this
alternative definition gives the same result as Definition
\ref{def:embsumpatch}.  This alternative definition will make the
proof of a generalization of Stallings' Theorem \ref{thm:stal} very
easy (see Theorem \ref{thm:genstal}).
\medskip

In the following definition we enrich the structure of cylindrical
cobordism of Definition \ref{def:cylcob}:

\begin{definition} \label{def:stiff}
   A {\bf stiffened cylindrical cobordism} (see Figure \ref{I}) is a cylindrical cobordism
   $W: M^- \Longmapsto M^+$ and a neighborhood $V$ (the {\bf stiffening}) of
   $M^-  \bigsqcup  M^+$ in $W$, endowed with a diffeomorphism to a
   neighborhood of $(\partial I) \times M$ in $I \times M$ of the form:
       \[   (I \setminus \mbox{int}(C)) \times M,  \]
   which extends the restriction to $V$ of the
   given diffeomorphism $\partial_{cyl} W \simeq \partial (I \times M)$.
   Here $C \hookrightarrow \mbox{int}(I)$ denotes a compact
    subsegment, called the {\bf core} of the stiffened cobordism.
    The pull-back to $V \cup \partial_{cyl} W$ of the first projection $I \times M \to I$ is
    called the {\bf height function} of the stiffened cylindrical cobordism.
\end{definition}

\begin{figure}[ht]
  \relabelbox \small {
  \centerline{\epsfbox{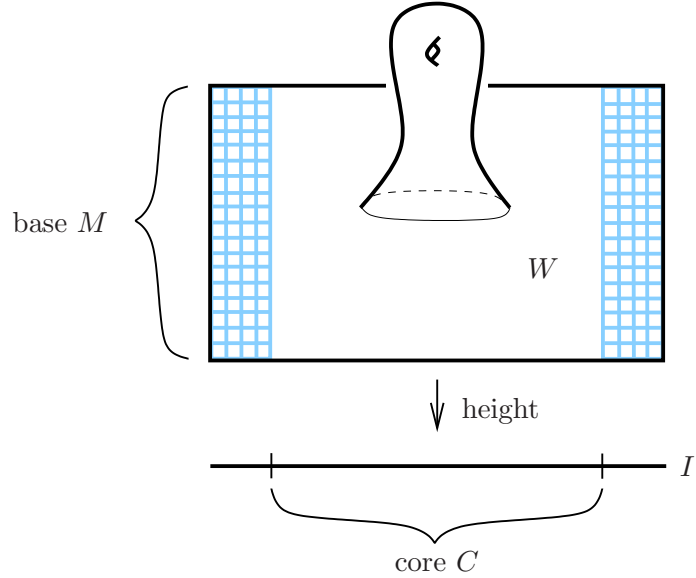}}}

\relabel{1}{$W$} \relabel{2}{base $M$} \relabel{6}{height}

\relabel{4}{core $C$} \relabel{5}{$I$}

\endrelabelbox
        \caption{A stiffened cylindrical cobordism $W$ with directing segment $I$}
        \label{I}
\end{figure}

\begin{remark}  \label{rem:stiffrem}
   \begin{enumerate}
        \item \label{rem:stiffisot}
             Given a cylindrical cobordism, stiffenings exist and are unique up to isotopy.

        \item Our choice of name is motivated by the fact that we see this supplementary
           structure as a way to rigidify or stiffen the initial cobordism.
     \end{enumerate}
\end{remark}

Recall from Lemma \ref{lem:obscyl} that one obtains cylindrical
cobordisms by splitting any manifold along a Seifert hypersurface.
Moreover, the two notions are equivalent, as shown by Proposition
\ref{prop:invop}. From this viewpoint, stiffenings correspond to
tubular neighborhoods of the Seifert hypersurface:

\begin{lemma}
     Let $M \hookrightarrow W$ be a Seifert hypersurface. Consider a
  collar neighborhood  $[-\theta, \theta] \times M$ of the strict transform
  $M \hookrightarrow \Pi_{\partial M} W$ of $M$ (see Definition \ref{def:polarblow}),
  which intersects  the boundary
  $ \bS^1 \times \partial M \hookrightarrow \Pi_{\partial M} W$
  along  $[-\theta, \theta] \times \partial M$. Here $\theta \in (0, \pi)$, therefore
  the segment  $[-\theta, \theta]$ is seen as an arc of the circle $\bS^1$.
  Then its image inside the
  splitting $\Sigma_M(W)$ is a stiffening of this cylindrical cobordism,
  with directing segment the splitting of $\bS^1$ at the point of argument $0$ and
  core segment $[\theta, 2 \pi - \theta]$.
\end{lemma}

A straightforward proof of this lemma easily follows  by
inspecting Figures \ref{Eblow} and \ref{Esplit}.

In the following definition we extend to stiffened cylindrical
cobordisms the notion of sum introduced for manifolds (see
Definition \ref{def:absum}) and for hypersurfaces (see Definition
\ref{def:embsumpatch}):

\begin{definition}  \label{def:sumstif}
     Consider two summable patched manifolds $(M_i, P)_{i = 1,2}$, with
     attaching regions $(A_i)_{i = 1,2}$.
     Let $(W_i : M_i^- \Longmapsto M_i^+, V_i)_{i =1,2}$ be two stiffened
     cylindrical cobordisms \emph{with identified directing segment} $I$.
     They are called {\bf summable} if their core intervals $(C_i)_{i =1,2}$ are disjoint
     and if $C_1$ is situated \emph{after} $C_2$ with respect to the orientation of $I$.
     In this case, their {\bf sum}, denoted by:
         \[ \displaystyle{(W_1, V_1) \biguplus^P \  (W_2, V_2)} \]
     is obtained by performing  the following operations using the stiffenings
        (see Figures \ref{Rbis} and \ref{R}):
          \begin{itemize}
           \item Over, $I \setminus (\mbox{int}(C_1) \cup \  \mbox{int}(C_2))$, sum
                 fiberwise $(M_1, P)$ to $(M_2, P)$ (that is, one has to multiply the gluing
                 map used to do this abstract sum by $\{t \}$, for any
                 $t \in I \setminus (\mbox{int}(C_1) \cup \  \mbox{int}(C_2))$).
\item   Over $C_1$, glue
                   $C_1 \times \overline{M_2 \setminus P}$ to $W_1$ along
                   $C_1 \times \partial M_1$ fiberwise (for each $t \in C_1$)
                     by the canonical identification of
                   $A_2 \hookrightarrow \partial ( \overline{M_2 \setminus P})$
                   and $A_2 \hookrightarrow \partial M_1$.
              \item Over $C_2$, glue
                   $C_2 \times \overline{M_1 \setminus P}$ to $W_2$ along
                   $C_2 \times \partial M_2$ fiberwise (for each $t \in C_2$)
                   by the canonical identification of
                   $A_1 \hookrightarrow \partial ( \overline{M_1 \setminus P})$
                   and $A_1 \hookrightarrow \partial M_2$.
\end{itemize}
\end{definition}

\begin{figure}[ht]
  \relabelbox \small {
  \centerline{\epsfbox{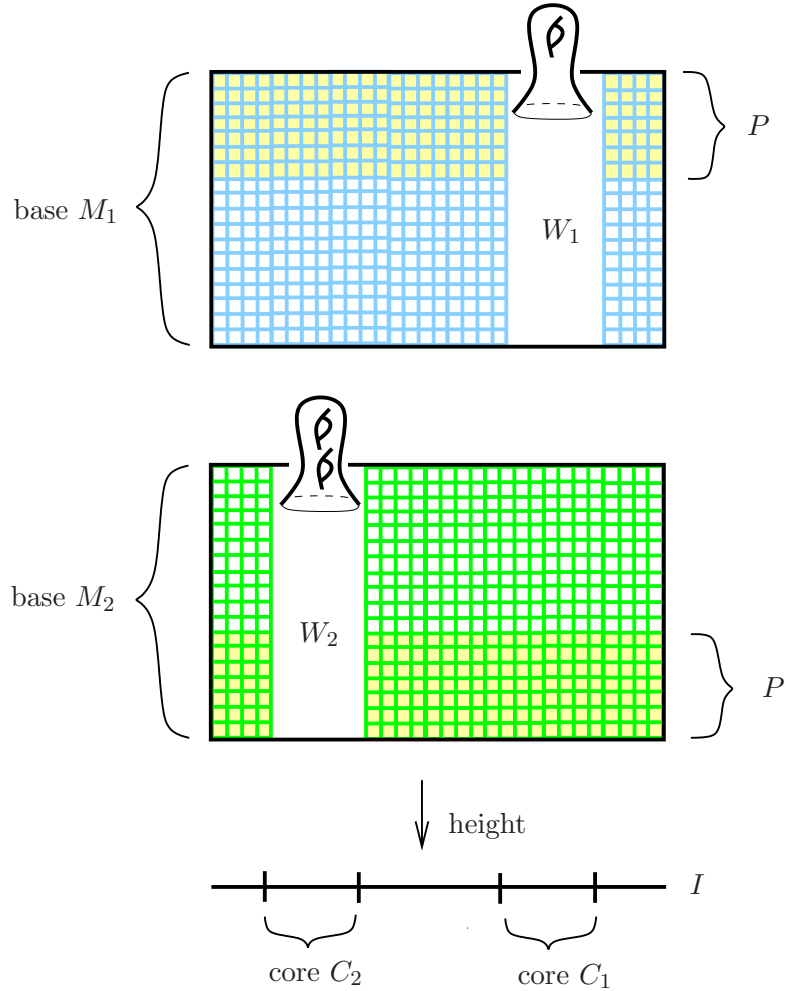}}}
\relabel{0}{$P$} \relabel{1}{$W_1$} \relabel{2}{base $M_1$}
\relabel{3}{base $M_2$} \relabel{4}{$W_2$} \relabel{5}{height}
\relabel{6}{$I$} \relabel{7}{core $C_2$} \relabel{8}{core $C_1$}
\relabel{9}{$P$}
\endrelabelbox
        \caption{Two summable stiffened cylindrical cobordisms}
\label{Rbis}
\end{figure}

\begin{remark}  \label{rem:stiffwell}
    \begin{enumerate}
         \item  \label{rem:twodef}
             Let $h_i$ denote the height
             function of the stiffened cylindrical cobordism $W_i$, for $i =
              1,2$. In Definition~\ref{def:sumstif}, we use the facts that for
            sufficiently small (and also sufficiently large) $t \in I$, the
           fiber $h_i^{-1}(t)$ is canonically identified with $M_i$ and that
             this identification extends to an identification of
             $h_i^{-1}(t) \cap
             \partial W_i$ with $\partial M_i$ for all $t \in I$, by the
             definition of a stiffened cylindrical cobordism.   All the gluings above fit
            together by  Remark \ref{rem:varem} (\ref{rem:welldef}).
        \item  \label{rem:natcyl}
             The sum $\displaystyle{W_1 \biguplus^P W_2}$ gets a natural structure of
             stiffened cylindrical cobordism with basis $\displaystyle{M_1 \biguplus^P M_2}$,
             directing segment  $I$ and core segment the convex hull inside $I$ of the cores
             $C_1$ and $C_2$. The new stiffening is the image inside
             $\displaystyle{W_1 \biguplus^P W_2}$ of the union of the initial stiffenings, and the
             two initial height functions glue into the new height function.
    \end{enumerate}
\end{remark}

Next, we extend the summing operation to cylindrical cobordisms
whose directing segments are not identified, and which do not have
fixed stiffenings. One has to make the following choices:

  \begin{itemize}
       \item Choose stiffenings. This choice is unique up to isotopy (see Remark
             \ref{rem:stiffrem}   (\ref{rem:stiffisot})).
        \item Identify their directing segments  by an orientation-preserving diffeomorphism.
              \end{itemize}

              There are two ways to make such an identification,
             up to isotopy, in order to guarantee the disjointness of the cores,
             which is an essential hypothesis
             in Definition \ref{def:sumstif}. Therefore, one gets an operation which is a priori
             non-commutative. The fact that it is indeed in general non-commutative
             results from the combination of propositions \ref{prop:propembsum} and
             \ref{prop:samedef}. More precisely, we use the fact, resulting from the proof
             of Proposition \ref{prop:propembsum} using any kinds of bands,
             that the embedded summing operation
             is non-commutative even when the hypersurfaces are globally
             cooriented.

\begin{definition}  \label{def:embsum}
     Consider two summable patched manifolds $(M_i, P)_{i = 1,2}$, with
     attaching regions $(A_i)_{i = 1,2}$.
     Let $(W_i : M_i^- \Longmapsto M_i^+)_{i =1,2}$ be two
     cylindrical cobordisms with directing segments $(I_i)_{i = 1,2}$.
     Choose stiffenings for both of them.
     Let $\varphi: I_1 \to I_2$ be an orientation-preserving  diffeomorphism
     which places the core segment of $I_1$ \emph{after} the core segment
     of $I_2$.  The {\bf sum} of $W_1$ and $W_2$, denoted by:
           \[ W_1 \biguplus^P W_2 \]
     is obtained by applying
     Definition \ref{def:sumstif} after identifying the directing segments $I_1$ and $I_2$
     using  the diffeomorphism $\varphi$.
\end{definition}

\begin{figure}[ht]
  \relabelbox \small {
  \centerline{\epsfbox{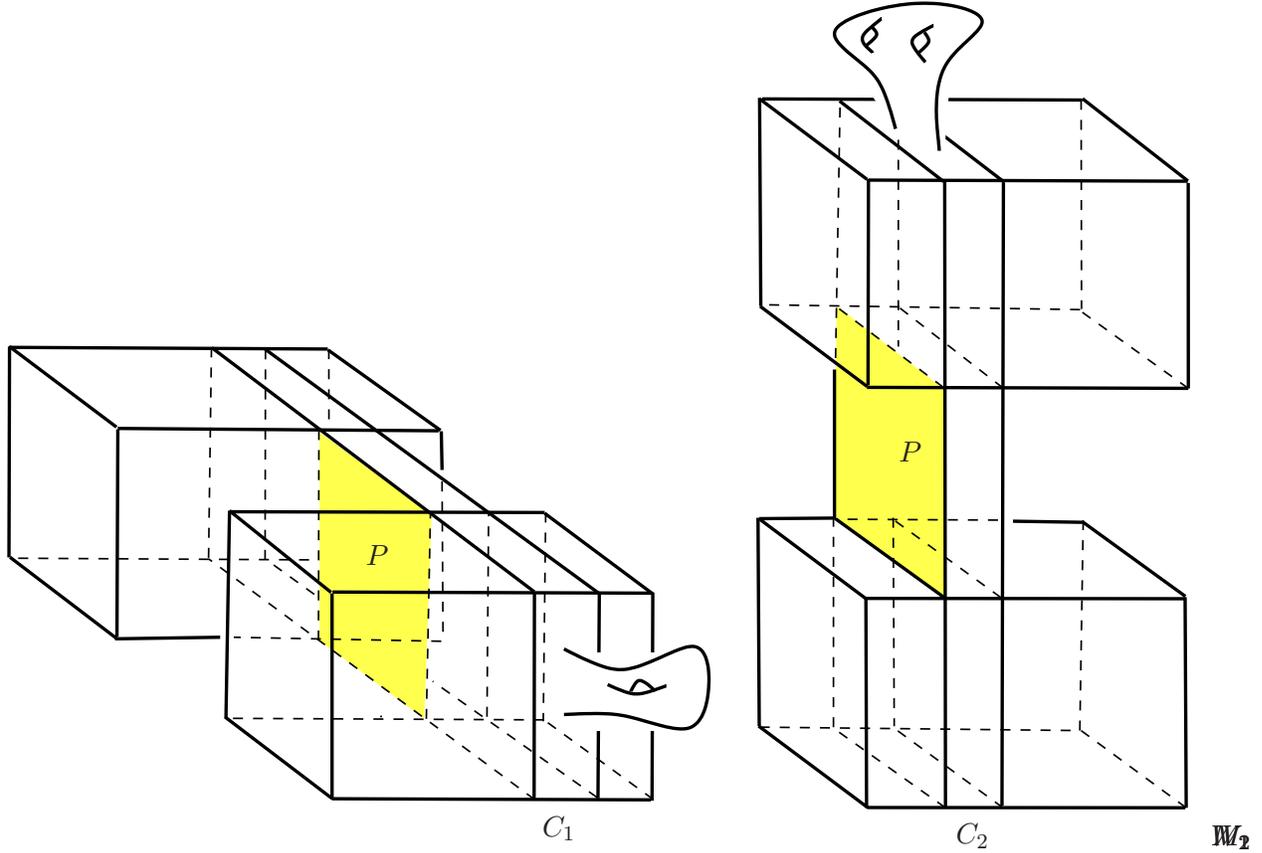}}}

\relabel{a}{$P$}

\relabel{b}{$P$}

\relabel{c}{$C_1$}

\relabel{d}{$C_2$}

\relabel{e}{$M_1$}

\relabel{f}{$M_2$}

\relabel{g}{$W_1$}

\relabel{h}{$W_2$}

\endrelabelbox
        \caption{The stiffened cylindrical cobordism $W_i$ (for $i=1,2$) is represented by the
        solid rectangular box where the solid green ball in the interior is removed. The sum
        $\displaystyle{(W_1, V_1) \biguplus^P \  (W_2, V_2)}$ will look like Figure~\ref{V},
        except that two disjoint
        solid balls have to be removed from the interior (colour figure online).}

        \label{R}
\end{figure}

\begin{remark}
    The diffeomorphism $\varphi$ which places the core segment of $I_1$ \emph{after}
    the core segment of $I_2$ being well-defined up to isotopy, as well as the stiffenings,
    we deduce that
    the sum is well-defined up to diffeomorphisms fixed on the cylindrical boundary
    of the cylindrical cobordism $\displaystyle{W_1 \biguplus^P W_2}$
    (see Remark \ref{rem:stiffwell} (\ref{rem:natcyl})).
\end{remark}

\bigskip
\section{Embedded summing is a natural geometric operation}
\label{sect:natgeomop}

In this section we prove an extension of Stallings'
Theorem~\ref{thm:stal} to arbitrary dimensions. Namely, we prove
that the embedded sum of two pages of open books is again a page of
an open book (see Theorem \ref{thm:genstal}). We extend this result
to pages of what we call \emph{Morse open books} (see Theorem
\ref{thm:genmorse}). A direct consequence of this theorem is a
generalization to arbitrary dimensions of a theorem proved in
dimension $3$ by Goda.  Both theorems illustrate Gabai's credo that
``Murasugi sum is a natural geometric operation''. Their proofs are
parallel and are based on the fact that, in the case of Seifert
hypersurfaces, the embedded sum as described in Definition
\ref{def:embsumpatch} may be equivalently described using the
operation of sum of cylindrical cobordisms described in Definition
\ref{def:embsum} (see Proposition \ref{prop:samedef}). Technically
speaking, this is the most difficult result of the paper.

\medskip

The following proposition shows that in the case in which one works
with summable patched Seifert hypersurfaces (see Definition
\ref{def:patchseif}), the previous notion of sum of cylindrical
cobordisms gives the same result as the embedded sum of two
patch-cooriented triples with identified patches:

\begin{proposition} \label{prop:samedef}
       Let $(W_1, M_1, P)$ and $(W_2, M_2, P)$ be
         two summable patched Seifert hypersurfaces.
         Then their embedded sum (see Definition  \ref{def:embsumpatch}):
            \[ \displaystyle{M_1 \biguplus^{P} \  M_2} \     \hookrightarrow
                 \  \   \displaystyle{(W_1,M_1) \biguplus^{P} \  (W_2, M_2)} \]
         is diffeomorphic,
         up to isotopy, to the Seifert hypersurface associated to the cylindrical cobordism
         (see definitions \ref{def:circol} and \ref{def:sumstif}):
             \[ \displaystyle{\Sigma_{M_1}(W_1) \biguplus^P  \Sigma_{M_2}(W_2)}. \]
\end{proposition}

     \begin{proof}
       We start from the cylindrical cobordisms $\Sigma_{M_1}(W_1)$ and
             $\Sigma_{M_2}(W_2)$, to which we apply Definition \ref{def:embsum}.
             We want to show that the associated Seifert hypersurface is diffeomorphic
             to that obtained using Definition \ref{def:embsumpatch}. In order to achieve this,
             we will show that the circle-collapsed mapping torus of
             $\displaystyle{\Sigma_{M_1}(W_1) \biguplus^P  \Sigma_{M_2}(W_2)}$
             may be obtained from the circle-collapsed mapping tori of
             the factors $\Sigma_{M_i}(W_i)$ by removing codimension $0$
             submanifolds which are diffeomorphic to $[0,1] \times P$, and identifying the resulting
             boundaries appropriately.

      The difficulty is that
             those submanifolds do not appear directly with the desired product
             structures, but as the unions of several codimension $0$ submanifolds.
             It turns out that all of them are endowed with product structures and those
             structures are related in a way which allows us to achieve our aim.

         Rather than  working with the circle-collapsed mapping tori
         $T_c(\Sigma_{M_i}(W_i))$,
   we will use instead the manifolds  obtained by filling the boundaries
   of the mapping tori $T(\Sigma_{M_i}(W_i))$ by the products $\bD^2 \times \partial M_i$.
   As stated in Lemma \ref{lem:secdescr}, those are simply different models of the same
   Seifert hypersurfaces.    Therefore, for $i=1,2$, we denote:
        \[ \Phi_{\partial M_i} (W_i) :=   \Pi_{\partial M_i} (W_i)
              \cup_{\bS^1 \times \partial M_i}   (\bD^2 \times \partial M_i),  \]
    where $\Pi_{\partial M_i} (W_i)$ is the result of piercing $W_i$ along $\partial M_i$
    (see Definition \ref{def:polarblow}) and the two manifolds-with-boundary on the
    right-hand-side are glued through the canonical identifications of their boundaries with
   $\bS^1 \times  \partial M_i$. Similarly, we will fill by a product the boundary of
    $\displaystyle{\Sigma_{M_1}(W_1) \biguplus^P  \Sigma_{M_2}(W_2)}$.

        We choose stiffenings $V_i$ of $\Sigma_{M_i}(W_i)$ and identifications of their
         directing segments that allow us to perform the sum as in Definition \ref{def:sumstif}.

       We may now apply the gluing operations described in the Definition
            \ref{def:sumstif} of the sum of stiffened cylindrical cobordisms with identified
            directing segments. Recall that over
            $I \setminus \ ( \mbox{int}(C_1) \cup \  \mbox{int}(C_2))$ those gluings may be described
            in several ways. The point here is to choose the description which is best adapted
            to our aim.

      Denote $\alpha_{\pm} := \partial_{\pm} I$ and choose a point $\beta \in I$ which
         lies strictly between the two cores $C_1$ and $C_2$. Denote  (see Figure~\ref{M}):
             \[ I_1 := [\alpha_{-}, \beta], \ \ \ I_2 :=   [\beta, \alpha_+].   \]

\begin{figure}[ht]
  \relabelbox \small {
  \centerline{\epsfbox{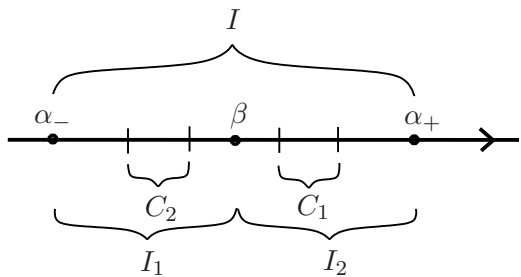}}}

\relabel{1}{$\alpha_-$} \relabel{2}{$\beta$}

\relabel{3}{$\alpha_+$} \relabel{6}{$C_2$} \relabel{7}{$C_1$}

\relabel{4}{$I_1$} \relabel{5}{$I_2$} \relabel{8}{$I$}

\endrelabelbox
        \caption{The interval $I$}
        \label{M}
\end{figure}

      We will do the gluings of Definition \ref{def:sumstif} by removing
             $P$ fiberwise from $\Sigma_{M_i}(W_i)$ over $I_i$, for each $i \in \{1,2 \}$.
             But we interpret the gluing operations directly on the mapping torus
             of $\Sigma_{M_i}(W_i)$. A  simple schematic representation of the 
             operation of summing
             $(M_1, P)$ and $(M_2, P)$ is depicted abstractly in
             Figure~\ref{N}, in order to help the reader following easily Figure~\ref{O}.
             We denote by $E_i$ the closure
             in $\partial M_i$ of $\partial M_i \setminus B_i$, where $B_i$ is the
             non-attaching region of $(M_i, P)$ (see Definition \ref{def:attastruc}),
             and by $K$ the closure of $\partial P \setminus (A_1 \cup A_2)$.

\begin{figure}[ht]
  \relabelbox \small {
  \centerline{\epsfbox{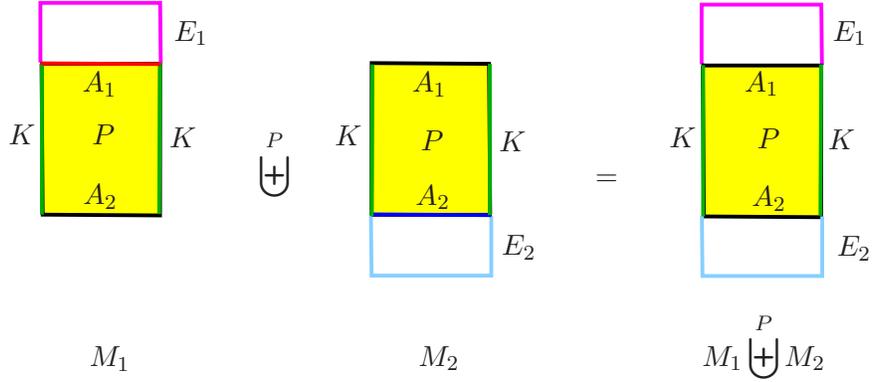}}}

\relabel{1}{$P$} \relabel{2}{$P$}

\relabel{3}{$P$} \relabel{6}{$A_1$} \relabel{p}{$A_1$}
\relabel{q}{$A_1$} \relabel{7}{$A_2$} \relabel{w}{$A_2$}
\relabel{r}{$A_2$}

\relabel{4}{$\displaystyle{\biguplus^{P}}$} \relabel{5}{$=$}

\relabel{8}{$K$} \relabel{x}{$K$} \relabel{y}{$K$} \relabel{z}{$K$}
\relabel{9}{$E_1$} \relabel{a}{$E_2$} \relabel{b}{$E_2$}
\relabel{c}{$E_1$}

\relabel{g}{$K$}

\relabel{h}{$K$}

\relabel{d}{$M_1$} \relabel{e}{$M_2$} \relabel{f}{$\displaystyle{M_1
\biguplus^{P} M_2}$}

\endrelabelbox
       \caption{The schematic representation of $E_i$ and $K$}
        \label{N}
\end{figure}

      The steps of the construction, interpreted using our filled models
            $\Phi_{\partial M_i} (W_i)$ of $(W_i, M_i)$, are:
            \begin{itemize}
                \item For each $i \in \{1, 2 \}$, remove $(I_i \times P) \cup
                       (\bD^2 \times \partial{M_i})$
                   from $\Phi_{\partial M_i}(W_i)$, then take the closure.
                \item Glue through the canonical identification
                    the portions of the resulting boundaries which are isomorphic to (see Figure~\ref{O}):
                      \[ (I_1 \times A_1) \cup (I_2 \times A_2) \cup (\alpha_{\pm} \times P) \cup
                                (\beta \times P). \]
                 \item Fill then the resulting boundary by:
                       \[ \bD^2 \times \partial(\ds{M_1 \biguplus^P \ M_2}) = (\bD^2 \times E_1) \cup
                      (\bD^2 \times K) \cup (\bD^2 \times E_2). \]
            \end{itemize}

\begin{figure}[ht]
  \relabelbox \small {
  \centerline{\epsfbox{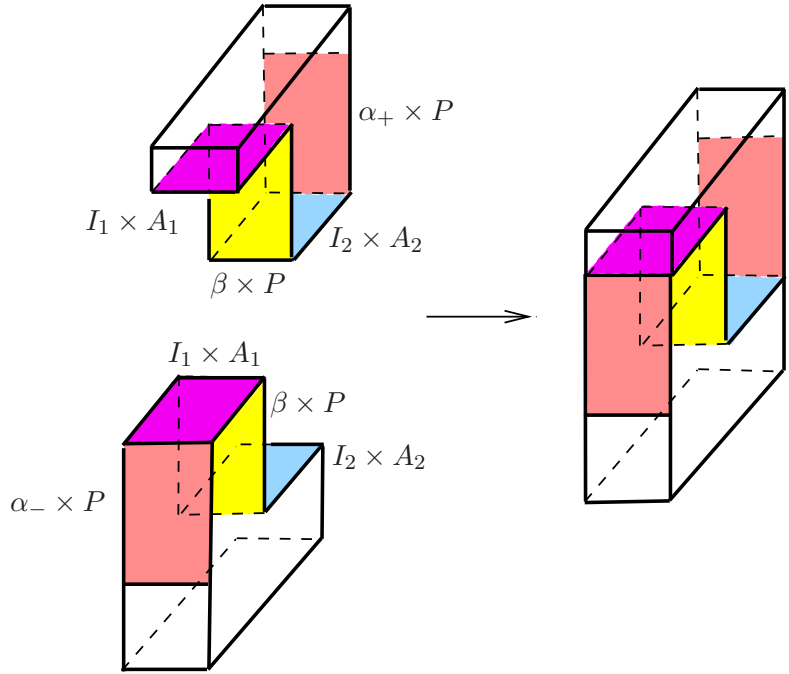}}}

\relabel{1}{$I_1 \times A_1$} \relabel{2}{$\beta \times P$}

\relabel{3}{$I_2 \times A_2$}

\relabel{4}{$\alpha_+ \times P$}

\relabel{8}{$I_1 \times A_1$} \relabel{6}{$\beta \times P$}

\relabel{7}{$I_2 \times A_2$}

\relabel{5}{$\alpha_- \times P$}

\endrelabelbox
        \caption{This Figure is to be compared with Figure \ref{V}}
        \label{O}
\end{figure}

  Note that  the pieces
  $\bD^2 \times E_i \hookrightarrow \bD^2 \times \partial M_i$ are first removed,
  then inserted back into the same position (that is, we glue exactly as before to the adjacent
  pieces). Therefore, we obtain the same final result without touching them.

        Instead, the piece $\bD^2 \times K$ is removed \emph{twice} and put back only
            \emph{once}. One may obtain the same result by cutting the disc $\bD^2$
            into two half-discs $\bD_1$ and $\bD_2$, as represented in Figure~\ref{P}, and
            only removing two conveniently chosen complementary half-discs. Namely,
            we will remove $\bD_i \times K$ from $\Phi_{\partial M_i} (W_i)$.

\begin{figure}[ht]
  \relabelbox \small {
  \centerline{\epsfbox{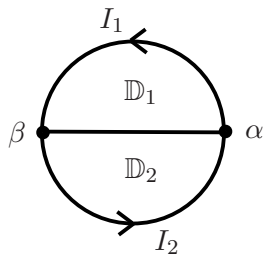}}}

\relabel{1}{$\beta$} \relabel{2}{$\alpha$}

\relabel{3}{$I_1$} \relabel{6}{$\bD_2$}

\relabel{4}{$I_2$} \relabel{5}{$\bD_1$}

\endrelabelbox
        \caption{Splitting the unit disk}
        \label{P}
\end{figure}

         The reinterpreted construction is:
                  \begin{itemize}
                \item Remove $(I_1 \times P) \cup (\bD^2 \times A_2) \cup (\bD_1 \times K)$
                   from $\Phi_{\partial M_1}(W_1)$, then take the closure. Symmetrically,
                   remove $(I_2 \times P) \cup (\bD^2 \times A_1) \cup (\bD_2 \times K)$
                   from $\Phi_{\partial M_2}(W_2)$, then take the closure.
                \item Glue the resulting boundaries through the canonical identification.
            \end{itemize}

            Notice now that $(I_1 \times P) \cup (\bD^2 \times A_2) \cup (\bD_1 \times K)$
            is isomorphic to $I_1 \times P$, and similarly
            $(I_2 \times P) \cup (\bD^2 \times A_1) \cup (\bD_2 \times K)$ is isomorphic
            to $I_2 \times P$. Indeed, in each case we may apply Lemma
            \ref{lem:grow} twice to end up with a description as in Definition
            \ref{def:embsumpatch}.
    \end{proof}

We leave the proof of the following intuitively clear lemma to the
reader:

\begin{lemma} \label{lem:grow}
  Let $Q$ be a manifold-with-boundary and $B \hookrightarrow \partial Q$
  be a full-dimensional submanifold-with-boundary of the boundary.
  Then the result of gluing $[0, 1] \times B$ to $Q$ through the canonical
  identification of $0 \times B$ with $B$ is isomorphic to $Q$ through
  an isomorphism which is the identity outside an arbitrarily small neighborhood
  of $B$ in $Q$.
\end{lemma}

Here is our generalization of Stallings' Theorem \ref{thm:stal}
(recall that the notion of open book was explained in Definition
\ref{def:openbook}):

\begin{theorem}   \label{thm:genstal}
   Let $(W_i, M_i, P)_{i = 1,2}$ be two summable patched Seifert hypersurfaces
   which are pages of open books on the closed manifolds $W_i$.
   Then the Seifert hypersurface associated to the sum
   $\displaystyle{(W_1, M_1) \biguplus^P  (W_2, M_2)}$
   is again a page of an open book. Moreover, the geometric monodromy
   of the resulting open book is the composition $\phi_1 \circ \phi_2$ of the monodromies
   of the initial open books. Here $\phi_i : M_i \to M_i$
   is extended to $\displaystyle{M_1 \biguplus^P M_2}$ by the identity on
   $(\displaystyle{M_1 \biguplus^P M_2} )\setminus M_i$.
\end{theorem}

    \begin{proof} Consider the splittings $\Sigma_{M_1}(W_1)$ and
         $\Sigma_{M_2}(W_2)$ of $W_1, W_2$ along the two pages.
         Let $(\partial M_i, \theta_i)$ be an open book on $W_i$
         such that $M_i= \theta_i^{-1}(0)$ (that is, such that $M_i$ is
         the page of argument $0$). The map $\theta_i: W_i \setminus \partial M_i \to \bS^1$
         lifts to an everywhere defined map $\tilde{\theta}_i : \Pi_{\partial M_i} W_i \to \bS^1$
         which is moreover a locally trivial fiber bundle projection. Therefore, it lifts
         to another fiber bundle projection:
            \[ \Sigma(\tilde{\theta}_i) : \Sigma_{M_i} W_i \to [0, 2 \pi] \]
           where the interval $[0, 2 \pi]$ is obtained by splitting the circle $\bS^1$ at the
           point of argument $0$.

          One may choose as stiffening of $\Sigma_{M_i} W_i$ a preimage
             $\Sigma(\tilde{\theta}_i)^{-1}([0, 2 \pi] \setminus \mbox{int}(C_i))$,
             where $C_i \subset (0, 2 \pi)$ is an arbitrary   compact segment with non-empty
             interior. Moreover, in order to get the hypothesis of Definition \ref{def:sumstif},
             we assume that $C_2$ and $C_1$ are disjoint and situated in this order
             on the segment
             $[0, 2 \pi]$ endowed with its usual orientation. One may take as height functions
             the projections $\Sigma(\tilde{\theta}_i)$ themselves.

             Definition \ref{def:sumstif} shows  that the two height functions glue
                into  a new globally defined height function:
                    \[ h:  \displaystyle{\Sigma_{M_1}(W_1) \biguplus^P  \Sigma_{M_2}(W_2) }\to [0, 2 \pi] \]
                which is again a fiber bundle projection. Its generic fiber is isomorphic to
                $\displaystyle{M_1 \biguplus^P M_2}$. Therefore, the associated
                Seifert hypersurface is again an open book, with page isomorphic to
                $\displaystyle{M_1 \biguplus^P M_2}$.

          But, by Proposition \ref{prop:samedef}, this Seifert hypersurface is
            isomorphic  to:
               \[  \displaystyle{M_1 \biguplus^{P}  M_2}    \hookrightarrow
                \displaystyle{(W_1,M_1) \biguplus^{P} \  (W_2, M_2)}. \]
             The proof of the last statement in the theorem is similar to the proof in the
             $3$-dimensional case (see Section \ref{sec:geomproof}).
    \end{proof}

Up to diffeomorphisms, all the choices of pages in an open book are
equivalent. Therefore, the previous theorem allows to define a
notion of \emph{sum} (generalized Murasugi sum) for open books:

\begin{definition}  \label{def:sumob}
    Assume that $(K_i, \theta_i)_{i = 1,2}$ are open book structures on the closed manifolds
$W_i$ of the same dimension.  Let $M_i$ be pages of them, and $P$ a common
patch of $M_1$ and $M_2$. Assume that $(M_1, P)$ and $(M_2, P)$ are summable.
The {\bf sum} of the two open books is the open book on
$\displaystyle{(W_1, M_1) \biguplus^P  \  (W_2, M_2)}$ constructed in the previous proof.
\end{definition}

The previous theorem may be extended to structures which are analogous to
open books, in the sense that they have bindings and are similar to open
books near them, but which are allowed to have Morse singularities away
from the bindings:

\begin{definition} \label{def:morseopenbook}
   A {\bf Morse open book} in a closed manifold $W$ is a pair  $(K,  \theta)$ consisting of:
\begin{enumerate}
     \item a codimension $2$ submanifold $K \subset W$, called the {\bf binding},
        with a trivialized normal bundle;
     \item  a map $\theta: W \setminus K \to \bS^1$ which, in a tubular neighborhood
         $\bD^2 \times K $ of $K$ is the normal angular coordinate,
         and which has only Morse critical points.
        The closure of any fiber $\theta^{-1}(\theta_0)$ is a {\bf page} of the Morse
        open book. A page is called {\bf regular} if $\theta_0$ is a regular value of
        $\theta$ and {\bf singular} otherwise.
 \end{enumerate}
\end{definition}

\begin{remark}
\begin{enumerate}
    \item The previous definition extends to arbitrary dimensions the notion of
      ``\emph{regular Morse map}'' introduced in dimension $3$ by Weber, Pajitnov
      and Rudolph in \cite{WPR 02}.
    \item The regular pages of Morse open books are Seifert hypersurfaces.
    Conversely, any Seifert hypersurface is a regular page of a Morse open
    book.  Therefore, the problem of defining and finding the minimal complexity
     of such a Morse open book arises naturally, which motivates the rest
     of this section.
    \item All the pages of a classical open book are diffeomorphic,  but this is certainly not
      true for a Morse open book which has a singular page.  Even if one considers 
      only the regular pages of a Morse open book, we may be sure that they are diffeomorphic  
      only if they are preimages of points which belong to the same connected
      component of the complement of the critical image of $\theta$ inside
      $\bS^1$.
\end{enumerate}
\end{remark}

One has the following extension to this setting of
Theorem \ref{thm:genstal}:

\begin{theorem}  \label{thm:genmorse}
    Let $(W_i, M_i, P)_{i = 1,2}$ be two summable patched Seifert hypersurfaces
   which are regular pages of Morse open books on the closed manifolds $W_i$.
   Then the Seifert hypersurface associated to the sum
    $\displaystyle{(W_1, M_1) \biguplus^P \  (W_2, M_2)}$
   is again a regular page of a Morse open book, whose multigerm of
   singularities is isomorphic to the disjoint union of the multigerms of
   singularities of the initial Morse open books.
\end{theorem}

  \begin{proof}
        One may reason along the same lines as in the proof of Theorem \ref{thm:genstal}.
        The difference is that one has to choose now the core intervals $C_i$ such that
        $\mbox{int}(C_i)$ contains the critical values of the maps $\Sigma(\tilde{\theta}_i)$.
        One does not touch the neighborhoods of the critical points of the two Morse maps,
        which ensures that the new set of singularities are the disjoint unions of the two initial
        sets of singularities.
  \end{proof}

Inspired by the \emph{Morse-Novikov number} attached to a Seifert
surface in \cite[Section 6]{WPR 02},  we introduce the following
invariants in order to measure how far a Seifert hypersurface is to
being a page of an open book:

\begin{definition} \label{def:Morsenumb}
      Let $M \hookrightarrow W$ be a Seifert hypersurface in the closed manifold $W$
      of dimension $w \geq 1$. For each  $k \in \{1, ..., w-1\}$, denote by
      $m_k(W, M)$ be the minimal number of critical points of index $k$ of a map
      $\theta : W \setminus \ \partial M \to \bS^1$ such that $(\partial M, \theta)$ is
      a Morse open book, and $M$ is a regular page.
      We call it the $k$-th {\bf Morse number} of $(W, M)$.
\end{definition}

As an immediate consequence of Theorem \ref{thm:genmorse}, we have:

\begin{proposition}
         Let $(W_i, M_i, P)_{i = 1,2}$ be two summable patched Seifert hypersurfaces
         in the closed manifolds $(W_i)_{i = 1,2}$ of the same dimension $w \geq 1$.
         Then:
             \[ m_k(\displaystyle{(W_1, M_1) \biguplus^P  \  (W_2, M_2)} ) \leq
                     m_k(W_1, M_1) + m_k(W_2, M_2)   \]
         for each  $k \in \{1, ..., w-1\}$.
\end{proposition}

As explained in the introduction of \cite{HR 03}, this theorem was
proved in dimension $3$ by Goda \cite{G 92}, under a different but
equivalent formulation.

\bigskip
\section{Questions related to contact topology and singularity theory}
\label{sec:ques}

We conclude this paper with a list of questions. Almost all of them concern the
sum of open books and its relations with singularity theory and contact topology.
That is why we recall briefly the basics of those relations, developing part of the
information given in Remark \ref{rem:eqfib} (\ref{rem:histob}).

\medskip

 Consider a germ of polynomial function $f: (\C^n, 0) \to (\C, 0)$
which has an isolated singularity at the origin. Let $\bS^{2n-1}(r)
\hookrightarrow \C^n$ be the Euclidean sphere of radius $r >0$
centered at the origin. The argument of $f$ is well-defined outside
the $0$-level of $f$. Look at the restrictions of both objects to
the sphere $\bS^{2n-1}(r)$:
    \[ K := f^{-1}(0) \cap \bS^{2n-1}(r), \ \  \ \  \theta :  \bS^{2n-1}(r) \setminus \  K \to \bS^1. \]
Milnor proved in \cite{M 68} that $(K, \theta)$ is an open book on $\bS^{2n-1}(r)$,
whenever $r$ is sufficiently small. This result was extended by Hamm \cite{H 71}
to \emph{holomorphic} functions $f$ with isolated singularity, defined on
any germ of complex analytic space $(X,0)$ which is non-singular in the complement
of the base point $0$. In this case, one replaces $\bS^{2n-1}(r)$ by the intersection $M(r)$
of $X$ with a sphere of sufficiently small radius $r$, centered at $0$, once $(X,0)$ was embedded
in some affine space $(\C^N, 0)$. For $r >0$ small enough,
one gets in this way open books $(K, \theta)$ on
$M(r)$. In \cite{CNP 06}, such open books originating in singularity
theory were called \emph{Milnor open books}.

In 2002 Giroux \cite{Gi 02} launched a program of study of contact
topology through open books. Namely, he described a particularly
adapted mutual position of a contact structure and an open book on
any closed $3$-dimensional manifold, saying that, in that case,
\emph{the open book supports the contact structure}. In fact, in
1975  Thurston and Winkelnkemper \cite{TW 75} proved that any open
book supports a contact structure. Conversely, Giroux showed that
any contact structure is supported by some open book. Moreover, he
proved that two open books which support the same contact structure
are stably equivalent, that is, one may arrive at the same open book
by executing finite sequences of Murasugi sums with positive Hopf
bands, starting from each one of the initial open books.

In the same paper, Giroux sketched an extension of this theory to
higher dimensions. In particular, he defined higher dimensional
analogs of supporting open books. In this case, if one wants to
construct a contact structure starting from an open book, one has to
enrich it with symplectic-topological structures. Namely, the pages
are to be \emph{Weinstein manifolds} (see the recent monograph
\cite{CE 12} for a detailed exploration of this notion), and there
should exist a geometric monodromy respecting in some sense the
Weinstein structure.

In 2006, the paper \cite{CNP 06} of Caubel, N\'emethi and the second
author
 related the two instances where open books appear naturally: singularity theory and contact topology. Note that there are
canonical contact structures on the manifolds $M(r)$, as they are
level sets of a strictly plurisubharmonic function (the square of
the distance to $0$) on the complex manifold ($X \setminus 0$). In
\cite{CNP 06}, it was  proved that the Milnor open book of any
function $f : (X,0) \to (\C, 0)$
with an isolated singularity at $0$ supports the canonical contact structure,
whenever the radius $r$ is sufficiently small. This
generalized an analogous result proved before by Giroux \cite{Gi 03}, 
for the case where $X$ is smooth and where instead of round
spheres, deformed ones are chosen  adapted to a given
holomorphic germ $f$ with isolated singularity.

\medskip
 Here are our questions:

\begin{enumerate}
     \item
         An open book is considered to be \emph{trivial} if its page is a
         smooth ball and its geometric monodromy is the identity. 
         We call an open book \emph{indecomposable} if it cannot be written
         in a non-trivial way as a sum of open books (see Definition
         \ref{def:sumob}).
         Find sufficient criteria of indecomposability.
    \item
       Find sufficient criteria on germs of holomorphic functions
       $f: (X,0) \to (\C, 0)$ with isolated singularity to define indecomposable open books.
    \item
        Find natural situations leading to triples $(X_i, f_i)_{1 \leq 1 \leq 3}$
       of isolated singularities and holomorphic functions with isolated singularities
       on them, such that the Milnor open book of $(X_3, f_3)$ is a sum of the
       Milnor open books of $(X_1, f_1)$ and $(X_2, f_2)$.
    \item
       Consider an open book and a contact structure supported by this open book on a closed
       manifold. Describe an adapted
       position of a patch inside a page, relative to the contact structure, allowing
       to extend the operation of sum of open books to a sum of open books
       which support contact structures. Also, prove an analog  of
       the following result using appropriate patches in higher dimensions:
\begin{theorem}[Torisu \cite{T 00}]\label{thm:tor}
        Let $\xi_i$ be the contact structure on a $3$-manifold $M_i$
       supported by the open book $(\Sigma_i, \phi_i)$, for $i=1,2$. Then the connected sum
      $(M, \xi)=(M_1, \xi_1) \# (M_2, \xi_2)$ is supported by the open
    book $(\Sigma, \phi)$, where $\Sigma$ is the Murasugi sum of
    $\Sigma_1$ and $\Sigma_2$ and $\phi=\phi_1 \circ \phi_2.$
\end{theorem}

Let us point out that Giroux proved a particular instance of
Theorem~\ref{thm:tor} for stabilizations of open books  in higher
dimensions.
       \item
       In analogy with Goda's results of \cite{G 92}, find lower bounds for the following
       difference of Morse numbers (see Definition \ref{def:Morsenumb}):
            \[ m_k(\displaystyle{(W_1, M_1) \biguplus^P  (W_2, M_2)} )  -
                     (m_k(W_1, M_1) + m_k(W_2, M_2) )\]
         whenever $(W_i, M_i)$ are Seifert
         hypersurfaces in closed manifolds of the same dimension.
\end{enumerate}

\medskip

\begin{thebibliography}{00}
   \bibitem{BNR 12} Borodzik M.; N\'emethi, A.; Ranicki A. \emph{Morse theory for manifolds with
         boundary.}  arXiv:1207.3066, to appear in Algebraic and Geometric Topology.

   \bibitem{B 72}  Browder, W. \emph{Surgery on simply-connected manifolds.}
      Ergebnisse der Mathematik und ihrer Grenzgebiete {\bf 65}. Springer-Verlag,
      New York-Heidelberg, 1972.

   \bibitem{CNP 06} Caubel, C.; N\'emethi, A.; Popescu-Pampu, P.  \emph{Milnor
         open books and Milnor fillable contact 3-manifolds.} Topology {\bf 45} (2006),
         no. 3, 673-689.

   \bibitem{CE 12} Cieliebak, K.; Eliashberg, Y. \emph{From Stein to Weinstein and back.
       Symplectic geometry of affine complex manifolds.}
       American Mathematical Society Colloquium Publications, {\bf 59}.
       American Mathematical Society, Providence, RI, 2012.

   \bibitem{CGHH 11} Colin, V.; Ghiggini, P.; Honda, K.; Hutchings, M.
        \emph{Sutures and contact homology I.} Geom. Topol. {\bf 15} (2011), no. 3,
         1749-1842.

   \bibitem{D 62-1} Douady, A. \emph{Vari\'et\'es \`a bord anguleux et voisinages tubulaires.}
       S\'eminaire Henri Cartan (Topologie diff\'erentielle), 14-\`eme ann\'ee, 1961/62, no. 1,
       11 pages.  Available at
         http://www.numdam.org/numdam-bin/fitem?id=SHC\_1961-1962\_\_14\_\_A1\_0

    \bibitem{D 62-2} Douady, A. \emph{Th\'eor\`emes d'isotopie et de recollement.}
       S\'eminaire Henri Cartan (Topologie diff\'erentielle), 14-\`eme ann\'ee, 1961/62, no. 2,
       16 pages. Available at
       http://www.numdam.org/numdam-bin/fitem?id=SHC\_1961-1962\_\_14\_\_A2\_0

        \bibitem{D 62-3} Douady, A. \emph{Arrondissement des ar\^etes.}
       S\'eminaire Henri Cartan (Topologie diff\'erentielle), 14-\`eme ann\'ee, 1961/62, no. 3,
       25 pages. Available at
       http://www.numdam.org/numdam-bin/fitem?id=SHC\_1961-1962\_\_14\_\_A3\_0

   \bibitem{DL 72} Durfee, A. H. ; Lawson, H. B.
    \emph{Fibered knots and foliations of highly connected manifolds.}
     Invent. Math.  {\bf 17} (1972), 203-215.


\bibitem{E 06} Etnyre, J. \emph{Lectures on open book decompositions and contact
    structures.}  Clay Math. Proc. 5 (the proceedings of the
    ``\emph{Floer Homology, Gauge Theory, and Low Dimensional Topology Workshop}''),
    2006, 103-141.

    \bibitem{G 83} Gabai, D. \emph{The Murasugi sum is a natural geometric operation.} In
     \emph{Low-dimensional topology.} (San Francisco, Calif., 1981), 131-143,
     Contemp. Math., {\bf 20}, Amer. Math. Soc., Providence, RI, 1983.

     \bibitem{G 83bis} Gabai, D.  \emph{Foliations and the topology of 3-manifolds.}
         J. Differential Geom. {\bf 18} (1983), 445-503.

     \bibitem{G 85}  Gabai, D. \emph{The Murasugi sum is a natural geometric operation. II.}
         Combinatorial methods in topology and algebraic geometry (Rochester, N.Y., 1982),
          93-100, Contemp. Math., {\bf 44}, Amer. Math. Soc., Providence, RI, 1985.


   \bibitem{G 86}  Gabai, D. \emph{Detecting fibred links in $S^3$.} Comment. Math. Helv.
               {\bf 61}  (1986),  no. 4, 519-555.

   \bibitem{Gi 02}  Giroux, E. \emph{G\'{e}om\'{e}trie de contact: de la dimension trois vers
        les dimensions sup\'{e}rieures.} (French) [Contact geometry: from
        dimension three to higher dimensions]  Proceedings of the
       International Congress of Mathematicians, Vol. II (Beijing, 2002),
        405-414, Higher Ed. Press, Beijing, 2002.

   \bibitem{Gi 03}  Giroux, E. \emph{Contact structures and symplectic
    fibrations over the circle.} Notes of the summer school
     {\it Holomorphic curves and contact topology}, Berder, June 2003.
  Available at:
            http://webusers.imj-prg.fr/$\tilde{}$ emmanuel.ferrand/publi/giroux1.ps

 \bibitem{GG 06} Giroux, E.; Goodman, N. \emph{On the stable equivalence of open
books in three-manifolds.} Geom. Topol.  10  (2006), 97-114.

\bibitem{G 92} Goda, H.  \emph{Heegaard splitting for sutured manifolds and Murasugi sum.}
         Osaka J. Math. {\bf 29} (1992), no. 1, 21-40.

 \bibitem{Go 03} Goodman, N. \emph{Contact structures and open books},
      PhD Thesis, University of Texas at Austin, 2003.


\bibitem{H 71} Hamm, H. \emph{Lokale topologische Eigenschaften
    komplexer R{\"a}ume.} Math. Ann.\textbf{191} (1971), 235-252.


    \bibitem{HR 03} Hirasawa, M., Rudolph, L. \emph{Constructions of Morse maps for knots
         and links, and upper bounds on the Morse-Novikov number.} ArXiv:math/0311134v1.

    \bibitem{H 76} Hirsch, M. \emph{Differential topology}. Springer-Verlag, 1976.

  \bibitem{H 63} Hirzebruch, F. \emph{The topology of normal singularities of an
     algebraic surface (d'apr\`es un article de D. Mumford).} S\'em. Bourbaki 1962/63,
     no.{\bf 250}, February 1963.


    \bibitem{HNK 71} Hirzebruch, F.; Neumann, W. D.; Koh, S. S.
          \emph{Differentiable manifolds and quadratic forms.}  Appendix II by W. Scharlau.
          Lecture Notes in Pure and Applied Mathematics, Vol. {\bf 4}.
          Marcel Dekker, Inc., New York, 1971.

    \bibitem{KN 77} Kauffman, L.; Neumann, W. \emph{Products of knots, branched
       fibrations and sums of singularities.} Topology {\bf 16} (1977), 369-393.

     \bibitem{L 85}  Lines, D. \emph{On odd-dimensional fibred knots obtained by plumbing
        and twisting.} J. London Math. Soc. (2) {\bf 32} (1985), no. 3, 557-571.


     \bibitem{L 86}  Lines, D. \emph{On even-dimensional fibred knots obtained by plumbing.}
     Math. Proc. Cambridge Philos. Soc.  {\bf 100}  (1986),  no. 1, 117-131.

\bibitem{L 87}  Lines, D. \emph{Stable plumbing for high odd-dimensional fibred knots.}
      Canad. Math. Bull.  {\bf 30}  (1987),  no. 4, 429-435.

\bibitem{M 59} Milnor, J. \emph{Differentiable manifolds which are homotopy spheres.}
     Mimeographed notes  (1959). Published for the first time in \emph{Collected papers
     of John Milnor III. Differential topology.} American Math. Soc. 2007, 65-88.

  \bibitem{M 64} Milnor, J. \emph{Differential topology.} In \emph{Lectures on modern
      mathematics II.} Edited by T.L. Saaty, Wiley, New York (1964), 165-183.
     Reprinted in \emph{Collected papers
      of John Milnor III. Differential topology.} American Math. Soc. 2007, 123-141.

\bibitem{M 68} Milnor, J.  \emph{Singular points of complex hypersurfaces.}
    Annals of Mathematics Studies {\bf 61} Princeton University Press, Princeton,
    N.J.; University of Tokyo Press, Tokyo 1968.


     \bibitem{M 61} Mumford, D.  \emph{The topology of normal singularities of
         an algebraic surface and a criterion for simplicity.} Inst. Hautes \'Etudes Sci. Publ. Math.
          No. {\bf 9} (1961), 5-22.

      \bibitem{M 63} Murasugi, K. \emph{On a certain subgroup of the group of an
         alternating link.} Am. J. Math. {\bf 85} (1963), 544-550.

       \bibitem{N 81} Neumann, W. D. \emph{A calculus for plumbing applied to the topology
           of complex surface singularities and degenerating complex curves.}
           Trans. Amer. Math. Soc. {\bf 268} (1981), no. 2, 299-344.

       \bibitem{NR 87} Neumann, W.; Rudolph, L. \emph{Unfoldings in knot theory.}
            Math. Ann. {\bf 278} (1987), 409-439.

       \bibitem{P 07} Popescu-Pampu, P. \emph{The geometry of continued fractions
           and the topology of surface singularities.} In  \emph{Singularities in geometry
           and topology} 2004, 119-195, Adv. Stud. Pure Math., {\bf 46}, Math. Soc. Japan,
           Tokyo, 2007.


     \bibitem{R 51} Riemann, B. {\it Grundlagen f\"ur eine allgemeine Theorie der
          Functionen einer  ver\"anderlicher komplexer Gr\"osse.} Inauguraldissertation,
          G\"ottingen, 1851. Traduction en Fran\c{c}ais : {\it Principes fondamentaux pour une
          th\'eorie g\'en\'erale des fonctions d'une grandeur variable complexe.}
          Dans {\it \OE uvres math\'ematiques de Riemann.}, trad. L. Laugel,
         Gauthier-Villars, Paris, 1898, 2-60. R\'e\'edition J. Gabay, Sceaux, 1990.
         English translation: {\it Foundations for a general theory of functions of a complex
         variable.} In {\it Bernhard Riemann, Collected Papers}, translated from the
          1892 edition by R. Baker, C. Christenson, H. Orde, Kendrick Press, Inc. 2004, 1-42.


      \bibitem{R 98} Rudolph, L. \emph{Quasipositive plumbing (constructions of
          quasipositive  knots and links, V)} Proc. A.M.S. {\bf 126}, No. 1 (1998), 257-267.


\bibitem{S 78} Stallings, J. R. \emph{Constructions of fibred knots and
       links.} Algebraic and geometric topology (Proc. Sympos. Pure Math.,
         Stanford Univ., Stanford, Calif., 1976), Part 2,  pp. 55-60, Proc.
        Sympos. Pure Math. {\bf XXXII}, Amer. Math. Soc., Providence, R.I., 1978.

\bibitem{T 72} Tamura, I. \emph{Spinnable structures on differentiable manifolds.}
       Proc. Japan Acad. {\bf 48} (1972), 293-296.


\bibitem{TW 75} Thurston, W. P.; Winkelnkemper, H. {\em On the existence of contact forms on
$3$-manifolds.} Proc. Amer. Math. Soc. 52 (1975), 345-347.


\bibitem{T 00} Torisu, I. \emph{Convex contact structures and fibered links in 3-manifolds.}
    Internat. Math. Res. Notices  2000,  no. 9, 441-454.

   \bibitem{WPR 02} Weber, C.; Pajitnov, A.; Rudolf, L.
          \emph{The Morse-Novikov number for knots and links.} (Russian. Russian summary)
          Algebra i Analiz {\bf 13} (2001), no. 3, 105-118; english translation in St. Petersburg
          Math. J. {\bf 13} (2002), no. 3, 417-426.

  \bibitem{W 73} Winkelnkemper, H. E. \emph{Manifolds as open books.}
      Bulletin of the  A. M. S. {\bf 79} No. 1 (1973), 45-51.

  \bibitem{W 98} Winkelnkemper, H. E. \emph{The history and applications of open books.}
     Appendix to A. Ranicki's book \emph{High-dimensional knot theory. Algebraic surgery
     in codimension 2.}  Springer Monographs in Mathematics. Springer-Verlag,
     New York, 1998, 615-626.

\end{thebibliography}
\end{document}